\documentclass[leqno]{article}  
\usepackage[T1]{fontenc}
\usepackage[utf8]{inputenc}
\usepackage{float}
\usepackage{upgreek}
\usepackage{amsfonts,amssymb,amsthm}
\usepackage[]{rotating,graphicx}
\usepackage{etextools}
\usepackage{mathtools} 
\usepackage{textcomp}
\usepackage{relsize}
\usepackage{bbold}
\usepackage[bbgreekl]{mathbbol} 
\usepackage{ifthen}
\usepackage{mathabx}
\usepackage{eucal}
\usepackage{indentfirst}
\usepackage{tikz}
\usetikzlibrary{calc,positioning}
\usetikzlibrary{arrows,decorations.pathmorphing,decorations.markings,backgrounds,positioning,fit,petri,arrows.meta} 
\usetikzlibrary{patterns, patterns.meta,math}
\usepackage[scr]{rsfso}
\usepackage{graphics}
\usepackage{bbm} 
\renewcommand{\top}{\text{Top}}
\newcommand{\sch}{\text{Sch}}
\newcommand{\psch}{\text{$\biop$Sch}}
\newcommand{\gsch}{\text{GSch}}

\newcommand{\gl}{\text{GL}}
\newcommand{\xspace}{\mathcal{X}}
\newcommand{\pid}{\mathfrak{p}}
\newcommand{\oneideal}{\mathfrak{a}}

\newcommand{\hookedrightarrow}[1]{ \tikz	\draw[{Hooks[right,length=5,width=6]}->] (0,0) -- (#1cm,0cm);	\; }

\newcommand{\xmapsfrom}[1]{
								\;\begin{tikzpicture}
												\draw [{Stealth[sep,length=0.8mm,width=1.2mm]}-{Bar[width=3pt]}] (0,0) -- (#1mm,0);
								\end{tikzpicture}\;
				}

\newcommand{\myemph}[1]{{\it{\bf{#1}}}}

\usepackage{anyfontsize}
\usepackage{stmaryrd}

\newcommand{\xtworightarrow}[2][]{%
				  \xrightarrow[#1]{#2}\mathrel{\mkern-14mu}\rightarrow
	}
\newcommand{\xtwoleftarrow}[2][]{%
				  \leftarrow\mathrel{\mkern -14mu}\xleftarrow[#1]{#2}
	}

\newcommand{\op}{\text{op}}
\makeatletter
\newcommand*{\rom}[1]{\expandafter\@slowromancap\romannumeral #1@}
\makeatother
\newcommand{\cring}{\text{CRing}}

\newcommand{\unit}{{\mathbf{1}}}

\newcommand{\bigo}{\mathcal{O} }

\newcommand{\F}{\mathbb{F}}

\newcommand{\Q}{\mathbb{Q}}
\newcommand{\R}{\mathbb{R}}
\newcommand{\C}{\mathbb{C}}
\newcommand{\N}{\mathbb{N}}

\makeatletter
\DeclareRobustCommand{\properideal}{\mathrel{\text{$\m@th\proper@ideal$}}}
\newcommand{\proper@ideal}{%
  \ooalign{$\lneq$\cr\raise.22ex\hbox{$\lhd$}\cr}%
}
\makeatother

\newcommand{\Z}{\mathbb{Z}}

\newcommand{\smallcoprod}[1]{  
				\underset{#1}{\amalg}
}

\newcommand{\limrightarrow}{\mathop{\underrightarrow{\text{lim}}}}

\newcommand{\oplussum}{\mathop{\oplus}}

\newcommand{\E}{\mathcal{E}}

\newcommand{\biop}{\mathscr{P}}
\newcommand{\bioq}{\mathscr{Q}}

\newcommand{\circleftarrow}{\,\overleftarrow{\circ}\,}
\newcommand{\circrightarrow}{\,\overrightarrow{\circ}\,}

\DeclareMathOperator{\End}{End} 
\DeclareMathOperator{\bio}{Bio} 
\DeclareMathOperator{\cbio}{CBio} 
\DeclareMathOperator{\ctbio}{C_TBio} 
\DeclareMathOperator{\ctprop}{C_TProp} 
\DeclareMathOperator{\prop}{Prop} 
\DeclareMathOperator{\cprop}{CProp} 
\DeclareMathOperator{\crig}{CRig} 

\DeclareMathSymbol{\LLCurly}{\mathrel}{mathb}{"CE}
\DeclareMathSymbol{\ggcurly}{\mathrel}{mathb}{"CF}

\newtheorem*{claim*}{Claim}
\newtheorem{theorem}{Theorem}

\newtheorem{definition}[equation]{Definition} 
\newtheorem{definition*}{Definition} 
\newtheorem{remark}[equation]{Remark}
\newtheorem*{remark*}{Remark}
\newtheorem{example}[equation]{Example}
\newtheorem{example*}{Example}

\newtheorem*{corollary*}{Corollary}

\numberwithin{equation}{section}
\title{Non-Additive Geometry and Frobenius Correspondences}
\author{Shai Haran \\ {haran@technion.ac.il} }
\date{}
\begin{document}
\maketitle
\begin{abstract}
  The usual language of algebraic geometry is not appropriate for
  Arithmetical geometry: addition is singular at the real prime. We
  developed two languages that overcome this problem: one replace
  rings by the collection of “vectors” or by bi-operads and another
  based on “matrices” or props. These are the two languages of
  \cite{H17}, but we omit the involutions which brings considerable
  simplifications. Once one understands the delicate commutativity
  condition one can proceed following Grothendieck footsteps exactly.
  The square matrices , when viewed up to conjugation, give us new
  commutative rings with Frobenius endomorphisms.
\end{abstract}
\section*{Introduction} \label{sec:0}
Our starting point is Andr\'e Weil's ``Roseta - Stone'': The analogy between number fields
and function fields. On the one hand we have the mysterious number fields - the finite
extension of the field of rational numbers $\Q$. On the other hand we have the function
fields: the finite extensions $K/\F_{q}(z)$ (resp. $K/\C(z)$) of the field of rational
functions $\F_{q}(z)$ (resp. in characteristic $0$: $\C(z)$), which correspond to smooth
projective curves $X_{K}$ (with a dominant map $z:X_{K}\to \mathbb{P}_{k}^{1}$, 
$k=\F_q$ or $\C$); the points of $X_{K}$ correspond to the
valuations of $K/k$. \\
When we compare the field of rational functions $k(z)$ with the field of rational numbers
$\Q$ we discover: 
\begin{equation*}
\begin{tikzpicture}[baseline=0mm]
				 \node at (74mm,30mm) {the ``point at infinity''};
				\node at (40mm,30mm) {$k\left( (\frac{1}{z})\right)\supseteq k\left[ |\frac{1}{z}| \right] $};
				\draw[{Hooks[right,length=5,width=6]}->] ($(18mm,23mm)-(0,3mm)$)--($(27mm,32mm)-(0,3mm)$); 
				\node at (14mm,18mm) {$k(z)$};
				\draw[{Hooks[right,length=5,width=6]}->] ($(20mm,21mm)-(0,3mm)$)--($(27mm,21mm)-(0,3mm)$); 
				\node at (51mm,18mm) {$k((z-\alpha))\equiv k\left[ |z-\alpha| \right]\left(
				\frac{1}{z-\alpha} \right)$};
				\node at (14mm,10mm) {\rotatebox{90}{\Huge $\subseteq$}};
				\node at (35mm,10mm) {\rotatebox{90}{\Huge $\subseteq$}};
				\node at (14mm,03mm) {$k[z]$};
				\draw[{Hooks[right,length=5,width=6]}->] ($(20mm,06mm)-(0,3mm)$)--($(27mm,06mm)-(0,3mm)$); 
				\node at (51mm,02mm) {$k[|z-\alpha|]\equiv \lim\limits_{\underset{n}{\leftarrow}}
				k[z]/(z-\alpha)^n$};
\end{tikzpicture} 
\end{equation*}
versus: \\
\begin{equation*}
\begin{tikzpicture}[baseline=0mm]
				\node at (75mm,30mm) {the ``real prime''};
				\node at (42mm,30mm) {$\R\supseteq \Z_{\R}\equiv [-1,1]$};
				\draw[{Hooks[right,length=5,width=6]}->] ($(18mm,23mm)-(0,3mm)$)--($(27mm,32mm)-(0,3mm)$); 
				\node at (14mm,18mm) {$\Q$};
				\draw[{Hooks[right,length=5,width=6]}->] ($(18mm,21mm)-(0,3mm)$)--($(27mm,21mm)-(0,3mm)$); 
				\node at (39mm,18mm) {$\Q_{p}\equiv \Z_{p}\left( \frac{1}{p} \right)$};
				\node at (14mm,12mm) {\rotatebox{90}{\Huge $\subseteq$}};
				\node at (31mm,12mm) {\rotatebox{90}{\Huge $\subseteq$}};
				\node at (15mm,06mm) {$\Z$};
				\draw[{Hooks[right,length=5,width=6]}->] ($(18mm,09mm)-(0,3mm)$)--($(27mm,09mm)-(0,3mm)$); 
				\node at (40mm,04mm) {$\Z_p\equiv \lim\limits_{\underset{n}{\leftarrow}}
				\Z/p^n$};
\end{tikzpicture}
\end{equation*}
The ``real integers'' $\Z_{\R}\subseteq \R$ are not closed under addition and so do not
form a Ring. The original Roseta stone contained three different languages talking about
one and the same reality, while in mathematics we have one and the same language - the
language of rings: addition and multiplication - and this language is talking about three
different realities; moreover: it is not the appropriate language -addition is problematic
at the real prime!
\section{Bio (perads)} \label{sec:1}
\begin{definition}
				A Bio (short for Bi-operad) is a pair $\biop=(\biop^{-},\biop^{+})$ of symmetric
				operads $\biop^{-}$ and $\biop^{+}$ ``acting'' on each other. 
				\label{def:1.1}
\end{definition}
Explicitly, we have sets $\biop^{-}(n)$, $\biop^{+}(n)$, $n\ge 1$, with an action
of the symmetric group $S_{n}$ on the right on $\biop^{-}$, and on the left on
$\biop^{+}$,
\begin{equation}
				\begin{aligned}
\begin{tikzpicture}[baseline=0mm]
				 \draw[->] (16mm,14.1mm) arc [start angle=-130, end angle=130, radius=1.5mm];
				 \node at (10mm,15mm) {$\biop^{-}(n)$};
				 \node at (10mm,8mm) {$p\mapsto p\circ \sigma$};
				 \node at (35mm,15mm) {$S_{n}$};
				 \node at (35mm,12mm) {\rotatebox{90}{$\in$}};
				 \node at (35mm,8.5mm) {$\sigma$};
				 \node at (60mm,15mm) {$\biop^{+}(n)$};
				 \draw[->] (52mm,14.1mm) arc [start angle=130, end angle=-130, radius=-1.5mm];
				 \node at (57mm,8mm) {$\sigma\circ q\mapsfrom q$};
\end{tikzpicture}
				\end{aligned}
\label{eq:1.2}
\end{equation}
We have \myemph{composition} maps
\begin{equation}
				\begin{array}[H]{l}
				\begin{array}[H]{lll}
								\biop^{-}(n)\times \biop^{-}(k_1)\times \cdots \times
								\biop^{-}(k_{n})\longrightarrow \biop^{-}(k_1+\cdots+k_n)    \\
								(p,p_1,\cdots, p_{n})\mapsto p\circ(p_{i}) 			
				\end{array}, \\\\
				\begin{array}[H]{l}
								\biop^{+}(k_1+\cdots +k_n) \longleftarrow \biop^{+}(k_1)\times \cdots \times
								\biop^{+}(k_n)\times\biop^{+}(n) \\
								(q_i)\circ q\mapsfrom (q_1,\cdots , q_n,q) 
				\end{array}
				\end{array}
				\label{eq:1.3}
\end{equation}
which are associative 
\begin{equation}
				(p\circ(p_{i}))\circ(p_{ij})=p\circ\left( p_{i}\circ (p_{ij}) \right)\;\; , \;\; 
				\left( (q_{ij})\circ q_i \right)\circ q = (q_{ij})\circ \left( (q_{i})\circ q \right)
				\label{eq:1.4}
\end{equation}
and unital 
\begin{equation}
				\begin{array}[H]{l}
				p\circ(1^{-},\cdots, 1^{-})=p\;\; , \;\; 1^{-}\in \biop^{-}(1), \\\\
				\biop^{+}(1)\ni 1^{+} \;\; , \;\; q=(1^{+},\cdots ,1^{+})\circ q
				\end{array}
				\label{eq:1.5}
\end{equation}
and "$S_{k_{1}}\times \cdots \times S_{k_n}\rtimes S_{n}\subseteq S_{k_1+\ldots +
k_n}$ - covariant." We have also \myemph{action} maps 
\begin{equation}
				\begin{array}[H]{l}
								\biop^{-}(k_{1}+\cdots + k_{n})\times \biop^{+}(k_1)\times\cdots
								\times\biop^{+}(k_{n})\rightarrow\biop^{-}(n), \\\\
								(p,q_{1},\ldots q_{n}) \mapsto p \circleftarrow (q_{i}), \\\\
								\biop^{+}(n)\leftarrow \biop^{-}(k_{1})\times \cdots \times
								\biop^{-}(k_n)\times \biop^{+}(k_{1}+\cdots +k_{n}), \\\\
								(p_i)\circrightarrow q\mapsfrom (p_1,\cdots , p_n,q) 
				\end{array}
				\label{eq:1.6}
\end{equation}
which are associative
\begin{equation}
				\begin{array}[H]{l}
				p\circleftarrow \left( (q_{ij}\circ q_{i}) \right) = \left( p\circleftarrow
				(q_{i j}) \right)\circleftarrow (q_{i}), \\\\ 
				(p_{i})\circrightarrow \left( (p_{ij})\circrightarrow q \right) = \left(
								p_{i}\circ (p_{ij}) \right)\circrightarrow q
				\end{array}
								\label{eq:1.7}
\end{equation}
unital 
\begin{equation}
				p\circleftarrow(1^{-})=p \quad , \quad q=(1^{+})\circrightarrow q
				\label{eq:1.8}
\end{equation}
$S_{k_1}\times \cdots \times S_{k_n}$ - invariant and $S_{n}$-covariant. \\
We also require \myemph{naturality} of the action maps, so we have 
\begin{equation}
				\left( p\circ (p_{i}) \right)\circleftarrow \left( q_{ij} \right)=p\circ \left(
				p_{i}\circleftarrow (q_{ij}) \right), 
				\label{eq:1.9}
\end{equation}
pictorially \\\\
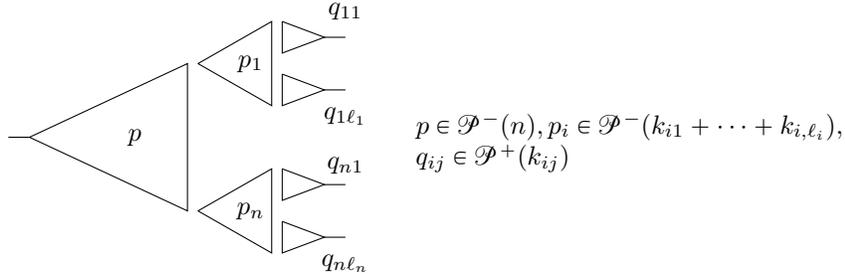
\begin{tikzpicture}[baseline=0mm, scale=1.4]
				 \draw[-] (3mm,15mm)--(5mm,15mm)--(20mm,8mm)--(20mm,22mm)--(5mm,15mm);
				 \draw[-] (21mm,8mm)--(28mm,12mm)--(28mm,4mm)--(21mm,8mm);
				 \draw[-] (21mm,22mm)--(28mm,26mm)--(28mm,18mm)--(21mm,22mm);
				 \draw[-] (29mm,26mm)--(29mm,23mm)--(33mm,24.5mm)--(29mm,26mm);
				 \draw[-] (33mm,24.5mm)--(35mm,24.5mm);
				 \draw[-] (29mm,21mm)--(29mm,18mm)--(33mm,19.5mm)--(29mm,21mm);
				 \draw[-] (33mm,19.5mm)--(35mm,19.5mm);
				 \draw[-] (29mm,12mm)--(29mm,9mm)--(33mm,10.5mm)--(29mm,12mm);
				 \draw[-] (33mm,10.5mm)--(35mm,10.5mm);
				 \draw[-] (29mm,7mm)--(29mm,4mm)--(33mm,5.5mm)--(29mm,7mm);
				 \draw[-] (33mm,5.5mm)--(35mm,5.5mm);
				 \node at (62mm,16mm) {$p\in\biop^{-}(n), p_{i}\in\biop^{-}(k_{i1}+\cdots+k_{i,\ell_i}),$};
				 \node at (49mm,13mm) {$ q_{ij}\in \biop^{+}(k_{ij})$};
				 \node at (15mm,15mm) {$p$};
				 \node at (26mm,8mm) {$p_n$};
				 \node at (26mm,22mm) {$p_1$};
				 \node at (35mm,27mm) {$q_{11}$};
				 \node at (35mm,17mm) {$q_{1\ell_1}$};
				 \node at (35mm,12.3mm) {$q_{n1}$};
				 \node at (35mm,3mm) {$q_{n\ell_{n}}$};
\end{tikzpicture}\\
and symmetrically, 
\begin{equation}
				\left( (p_{ij})\circrightarrow q_{i} \right)\circ q = \left( p_{ij}
				\right)\circrightarrow \left( (q_{i})\circ q \right);
				\label{eq:1.10}
\end{equation}
We also require co-naturality, so
\begin{equation}
				\big(p\circ (p_{ij})\big)\circleftarrow (q_{i})= p\circleftarrow \left(
				(p_{ij})\circrightarrow q_{i} \right)
				\label{eq:1.11}
\end{equation}
pictorially \\\\
\begin{tikzpicture}[baseline=0mm, scale=1.5]
				 \def\tri#1#2#3{
								 \draw[-] #1--#2--#3--#1;
				 }
				 \node at (15mm,15mm) {$p$};
				 \node at (23mm,24mm) {$p_{11}$};
				 \node at (23mm,16.5mm) {$p_{1\ell_1}$};
				 \node at (23mm,14mm) {$p_{n 1}$};
				 \node at (23mm,6mm) {$p_{n {\ell_n}}$};
				 \node at (23mm,20.75mm) {$\vdots$};
				 \node at (23mm,10.75mm) {$\vdots$};
				 \tri{(21mm,22mm)}{(25mm,23mm)}{(25mm,21mm)};
				 \tri{(21mm,18mm)}{(25mm,19mm)}{(25mm,17mm)};
				 \tri{(21mm,12mm)}{(25mm,11mm)}{(25mm,13mm)};
				 \tri{(21mm,8mm)}{(25mm,7mm)}{(25mm,9mm)};
				 \def\tritail#1#2#3#4{ 
								 \draw[-] #1--#2--#3--#1--($(#4mm,0mm)+#1$);
				 }
				 \tritail{(5mm,15mm)}{(20mm,8mm)}{(20mm,22mm)}{-2};
				 \tritail{(34mm,20mm)}{(27mm,23mm)}{(27mm,17mm)}{2};
				 \node at (30mm,20mm) {$q_1$};
				 \node at (27.6mm,16mm) {$\vdots$};
				 \tritail{(5mm,15mm)}{(20mm,8mm)}{(20mm,22mm)}{-2};
				 \tritail{(34mm,10mm)}{(27mm,13mm)}{(27mm,7mm)}{2};
				 \node at (30mm,10mm) {$q_n$};
				 \node at (52mm,17mm) {$p\in\biop^{-}(\ell_1+\cdots+\ell_n),\;$};
				 \node at (60mm,13mm){$p_{ij}\in\biop^{-}(k_{ij}),\;
				 q_{i}\in\biop^{+}(k_{i{1}}+\cdots+k_{i\ell_i})$};
\end{tikzpicture}\\
and symmmetrically, 
\begin{equation}
				\left( p_{i}\circleftarrow (q_{ij}) \right)\circrightarrow q =
				(p_{i})\circrightarrow \left( (q_{ij})\circ q \right)
				\label{eq:1.12}
\end{equation}
Maps of bios $\varphi:\biop\to \bioq$ are pairs of maps of symmetric operads
$\varphi^{\pm}: \biop^{\pm}\to \bioq^{\pm}$, preserving the $S_n$-actions, the
compositions and units, that also preserve the action maps; thus we have a category Bio. \\
In this definition we took the operads $\biop^{\pm}$ to be open. We can take them also to be closed, so that
we have also $\biop^-(0)=\{0^-\}$ and $\biop^+(0)=\{0^+\}$, and multiplication by $0^{\pm}$ at any place
$j\in \{1,\ldots, n\}$ map $\biop^{\pm}(n)$ to $\biop^{\pm}(n-1)$.
Similarly we can let $0^{\mp}$ act on $\biop^{\pm}(n)$, at any place $j$, mapping it to $\biop^{\pm}(n+1)$. With these $0^{\pm}$, and
with an involution ($\biop^{-}\xleftrightarrow{\;\sim\;}\biop^{+}$), we get a structure equivalent to the one of \cite{H17}. 

\begin{remark}
				The monoids $\biop^{-}(1)$ and $\biop^{+}(1)$ are connonically isomorphic via 
				\begin{equation*}
								\biop^{-}(1) \xleftrightarrow{\quad\sim\quad} \biop^{+}(1)
				\end{equation*}
				\begin{equation*}
								f^{-}\xmapsto{\qquad\quad} f^{-}\circrightarrow 1^{+}
				\end{equation*}
				\begin{equation*}
								1^{-1}\circleftarrow f^{+} \xmapsfrom{13} f^{+}
				\end{equation*}
				We use this as an identification: It identifies $1^{-}$ with $1^{+}$, and for
				$p\in \biop^{-}(n)$, $q\in\biop^{+}(n)$, $p\circrightarrow q$ is identified with
				$p\circleftarrow q$. \\ 
				We will abuse notations and write $\circ$ for either
				$\circleftarrow$ or $\circrightarrow$. 
				\label{remark:1.13}
\end{remark}
\begin{example}
				Given a symmetric monoidal category $\E$ and an object $X\in \E^{\circ}$, we have
				the bio $\End_{\E}(X)$ with 
				\begin{equation*}
								\End_{\E}^{-}(X)(n)= \E(X^{\otimes n},X) \quad , \quad
								\End_{\E}^{+}(X)(n)\equiv \E(X,X^{\otimes n})
				\end{equation*}
				with all the compositions and actions given by the monoidal structure of
				$\E$. 
				\label{example:1.14}
\end{example}
\begin{remark}
				Given a symmetric monoidal category $\E$, we can replace Set everywhere by
				$\E$ and obtain the category $\bio(\E)$ of bios in $\E$.
				\label{remark:1.15}
\end{remark}
\begin{remark}
				The category $\bio$ has an involution 
				\begin{equation*}
								\biop = (\biop^{-},\biop^{+})\xmapsto{\qquad\quad} \left(
								(\biop^{\op})^{-}\equiv \biop^{+},(\biop^{\op})^{+}\equiv \biop^{-}
				\right).
				\end{equation*}
				It takes the monoid $\biop(1)$ to the opposite monoid
				$\biop^{\circ p}(1)\equiv\biop(1)^{\op}$.\\
				We get the category $\bio^{t}$, with
				objects the bios $\biop$ with an involution 
				\begin{equation*}
				p\mapsto p^{t}:
				\biop\xrightarrow{\sim}{\quad}\biop^{\op},
				\end{equation*}
				and with maps preserving the involution. 
				\label{remark:1.16}
\end{remark}
\section{Rigs} \label{sec:2}
\begin{definition}
				A \myemph{Rig} ($\equiv$ Ring without Negatives) is a set $R$ with two
				operations of addition and multiplication 
				\begin{equation}
								R\times R \rightrightarrows R \quad , \quad (x,y)\xmapsto{\quad}
								\begin{array}[H]{l}
												x+y\\
												x\cdot y
								\end{array},
								\label{eq:2.1}
				\end{equation}
				both associative and unital 
				\begin{equation}
								\begin{array}[H]{lll}
												(x+y)+z=x+(y+z)\quad , \quad (x\cdot y)\cdot z = x\cdot (y\cdot z)
												\\\\
												x+0=x=0+x \quad , \quad x\cdot 1=x=1\cdot x
								\end{array}
								\label{eq:2.2}
				\end{equation}
				with addition always commutative $x+y=y+x$, and distributive 
				\begin{equation}
								\begin{array}[H]{l}
								(x_1+x_2)\cdot y = (x_1\cdot y)+(x_2\cdot y)\quad , \quad
								x\cdot(y_1+y_2)=(x\cdot y_1)+(x\cdot y_2) \\\\
								x\cdot 0 = 0 = 0\cdot x
								\end{array}
								\label{eq:2.3}
				\end{equation}
				\label{def:2.4}
\end{definition}
A rig with involution is a rig $R$ with an involution 
\begin{equation}
				\begin{array}[H]{l}
				x\mapsto x^{t}: R\righttoleftarrow \quad , \quad \text{satisfying}\; x^{tt}=x \;
				\text{and}\quad 0^{t}=0\quad , \quad 1^{t}=1 \\\\
				(x+y)^{t}= x^{t}+y^{t}\quad ,\quad (x\cdot y)^{t}=y^{t}\cdot x^{t}
				\end{array}
				\label{eq:2.5}
\end{equation}
A rig is commutative $R\in C\text{Rig}$ if multiplication is commutative,
\begin{equation*}
x\cdot y=y\cdot x 
\end{equation*}
A map of rigs $\varphi: R\to R^{\prime}$ is a set map preserving the operations and
constants 
\begin{equation}
				\varphi(0)=0 \quad , \quad \varphi(1)=1\quad ,\quad  
				\varphi(x+y)=\varphi(x)+\varphi(y)\quad , \quad \varphi(x\cdot y)=\varphi(x)\cdot
				\varphi(y)
				\label{def:19.6}
\end{equation}
and in the self-dual case $\text{Rig}^{t}$, $\varphi$ should also preserve the involution, 
\begin{equation*}
				\varphi(x^{t})=\varphi(x)^{t}. 
\end{equation*}
Thus we have categories and functors
\begin{align}
				\begin{array}[H]{l}
				 \begin{tikzpicture}
								 \node at (10mm,30mm) {$\text{Rig}^{t}$};
								 \node at (40mm,30mm) {$\text{Rig}$};
								 \node at (40mm,15mm) {$C\text{Rig}$};
								 \node at (40mm,22.5mm) {\rotatebox{90}{\Huge $\subseteq$}};
								 \draw[->] (15mm,30mm)--(35mm,30mm);
								 \node at (25mm,32mm) {$U$};
								 \draw[<-] (15mm,27mm)--(35mm,17mm);
				\end{tikzpicture}
				\end{array}
				\label{eq:2.7}
\end{align}
(for a commutative rig the identity is an involution). \vspace{.3cm}\\
There is a similar diagram  with
Ring instead of Rig. The inclusion $\text{Ring}\hookrightarrow \text{Rig}$ has the left
adjoint functor $K=\text{Grothendieck functor localizing addition}$. \vspace{.1cm}\\ 
E.g. We have the commutative
rigs
\begin{equation}
				\mathscr{B}=\left\{ 0,1 \right\}\hookedrightarrow{1} \mathscr{I}=[0,1]
				\hookedrightarrow{1} \mathscr{R}=[0,\infty]
				\label{eq:2.8}
\end{equation}
with the usual multiplication $x\cdot y$, and with the ``tropical'' addition
\begin{equation}
				x+y:\overset{def}{=}\max\left\{ x,y \right\}.
				\label{eq:2.9}
\end{equation}
\vspace{.1cm}
For $A\in \text{Rig}$ we have the operad $\biop_{A}$ 
with $\biop_{A}^{-}(n)\equiv A^{n}$, and with compositions 
\begin{equation}
				o_{i}: A^{n}\times A^{m} \xrightarrow{\qquad} A^{n+m-1}
				\label{eq:2.10}
\end{equation}
\begin{equation*}
				(a_1,\cdots , a_{i}, \cdots a_{n})\circ_{i}(b_1,\cdots , b_{m}):= (a_1,\ldots,
				a_{i-1},a_{i}\cdot b_{1},a_{i}\cdot b_2 \cdots a_i\cdot b_{m},a_{i+1}\cdots
a_{m})
\end{equation*}
We have similarly the operad $\biop_{A}^{+}\equiv \biop_{A^{\text{op}}}$ of column vectors 
$\biop_{A}^{+}(n)\equiv A^{n}$, with composition 
\begin{equation}
				\begin{array}[H]{c}
								A^{n}\times A^{m}\xrightarrow{\quad} A^{m+n-1}, \;\;
				\begin{pmatrix}
								a_1 \\ \vdots \\ a_n
				\end{pmatrix} \circ_j 
				\begin{pmatrix}
								b_1 \\ \vdots \\ b_{j}\\ \vdots \\ b_m
				\end{pmatrix} := 
				\begin{pmatrix}
								b_1 \\ \vdots \\ a_1\cdot b_{j} \\ \vdots \\ a_n\cdot b_{j} \\  \vdots \\
								b_{m}
				\end{pmatrix}
				\end{array}
				\label{eq:2.11}
\end{equation}
Moreover, we have mutual actions, for $1\le k\le m\le n$, 
\begin{equation}
				\begin{array}[H]{c}
								\overleftarrow{\circ}_{k}: A^{n}\times A^{m-k+1} \xrightarrow{\qquad} A^{n-m+k} \\\\
\begin{pmatrix}
				a_1\cdots a_k\cdots a_m\cdots a_n
\end{pmatrix} \circleftarrow_{k} 
\begin{pmatrix}
				b_0\\ b_1 \\ \vdots \\ b_{m-k} 
\end{pmatrix} := 
\\
\begin{pmatrix}
				a_1, \cdots a_{k-1}, a_{k}\cdot b_0+\cdots + a_{k+j}\cdot b_j+\cdots +
				a_{m}\cdot b_{m-k}, a_{m+1},\cdots a_n
\end{pmatrix}
				\end{array}
				\label{eq:2.12}
\end{equation}
and similarly for $\circrightarrow$. \\
This construction gives a full and faithful embedding 
\begin{equation}
				\text{Rig}\; \hookedrightarrow{1} \text{Bio} \quad , \quad
				A\xmapsto{\qquad} \biop_{A}.
				\label{eq:2.13}
\end{equation}
Moreover, when the rig $A$ has an involution, the bio $\biop_{A}=\left(
\biop^{-}_{A},\biop^{+}_{A} \right)$ has an involution
\begin{equation}
				(a_1,\cdots , a_{n})^{t} := 
				\begin{pmatrix}
								a_1^{t}\\ \vdots \\ a_{n}^{t}
				\end{pmatrix}
				\label{eq:2.14}
\end{equation}
giving 
\begin{equation}
\text{Rig}^{t} \hookedrightarrow{2}\text{Bio}^{t}
\label{eq:2.15}
\end{equation}
\begin{remark}
				For any rig $A$ the associated bio $\biop_A$ contains the sub-bio \break $\F\subseteq
				\biop_A$ with 
				\begin{equation}
								\begin{array}[H]{ll}
								\mathbb{F}^{-}(n) &= \left\{ 0=(0,\cdots , 0), \delta_{1}=(1,0,\cdots , 0),
								\cdots , \delta_{n}=(0,\cdots 0,1) \right\} \\\\
								\mathbb{F}^{+}(n) &= \left\{ 0^{t}=
																\left(\begin{matrix}
																								0 \\ \vdots \\ 0
																				\end{matrix}\right) , 
																				\delta_{1}^{t} = 
																								\left(\begin{matrix}
																								1 \\ 0 \\ \vdots \\ 0
																				\end{matrix}\right) ,  \cdots , 
																				\delta_{n}^{t}=
																				\left(\begin{matrix}
																								0 \\ \vdots \\ 0 \\ 1
																				\end{matrix}\right)
								\right\}
								\end{array}
								\label{eq:2.16}
				\end{equation}
				the (co)-vectors with at most one $1$ coordinate (note that $\mathbb{F}$ is closed
				under the composition and action operations); we call $\mathbb{F}$
				``\myemph{the field with one element}''. 
				\label{remark:2.16}
\end{remark}
\section{The $\ell_{p}$ bios and the real and complex ``integers''} \label{sec:3}
Fix $p,q\in [1,\infty]$, with $1/p + 1/q =1$. \\ 
We have the operad
$\biop_{\ell_p}\subseteq \biop_{\R}^{-}$, with 
\begin{equation}
        \begin{array}[H]{l}
		\biop_{\ell_p}(n):= \left\{ (x_1,\cdots, x_{n})\in \biop_{\R}(n)\equiv \R^{n},
		|x_1|^{p}+\cdots + |x_{n}|^{p}\le 1 \right\} \\\\
		(x_1,\cdots ,x_n)\circ_i (y_1,\cdots , y_m):=(x_1,\cdots, x_{i-1},x_{i}\cdot
		y_1,\ldots , x_i\cdot y_m,x_{i+1},\cdots , x_n)
	\end{array}
	\label{eq:3.1}
\end{equation}
We write the elements of $\biop_{\ell_q}$ as column vector, and we have the mutual actions
induced from $\biop_{\R}$,
\begin{equation}
				(a_1\cdots a_{k}\cdots a_n)\circleftarrow_{k}
				\begin{pmatrix}
								b_0\\b_1\\\vdots\\b_{m-k}				
				\end{pmatrix}
				:= (a_1\cdots a_{k-1},a_k\cdot b_0+\cdots
								+a_{k+j}\cdot b_{j}+\cdots + a_{m}\cdot b_{m-k},a_{m+1},\cdots ,
				a_{n})
				\label{eq:3.2}
\end{equation}
and similarly for $\circrightarrow$. Thus we have a $\text{sub}\text{-bio}\,(\biop_{\ell_p},\biop_{\ell_q})\subseteq \biop_{\R}$. \\
The rig $\R$, being commutative has an (identity) involution, hence $\biop_{\R}$ has an
involution 
\begin{equation}
				(a_1\cdots a_n)^{t}:=
				\begin{pmatrix}
								a_1\\\vdots\\ a_n
				\end{pmatrix}
				\label{3.3}
\end{equation}
The sub-bio $\left( \biop_{\ell_{p}}, \biop_{\ell_{q}} \right)\subseteq \left(
				\biop_{\R}^{-},\biop_{\R}^{+}
\right)$ is stable under this involution iff $p=q=2$, and we are in the self-dual case
\begin{equation}
				\biop_{\ell_2}(n)\equiv \left\{ (x_1,\cdots , x_{n})\in \R^n, x_1^{2}+\cdots
				+x_n^{2}\le 1 \right\}
				\label{eq:3.4}
\end{equation} the $\ell_2$-unit ball. \vspace{.1cm}\\
\noindent The bio $\mathbb{Z}_{\mathbb{R}}:=(\biop_{\ell_2},\biop_{\ell_2})\subseteq \biop_{\R}$ is the ``Real integers'',
analogue at the ``Real prime'' of the $p$-adic integers $\Z_{p}\subseteq
\Q_{p}$, $p$ a prime. We similarly have the ``complex integers'' $\Z_{\C}\subseteq
\biop_{\C}$ given by the unit $\ell_{2}$ complex balls. \\
Contracting the inside of the unit $\ell_{2}$-balls to a point we get quotient bios 
\begin{equation}
				\begin{array}[H]{lll}
								\Z_{\R}\xtworightarrow{\quad} \F_{\R}\quad &, \quad & \F_{\R}^{\pm}(n)\equiv
				S^{n-1}\sqcup \left\{ 0 \right\} \\\\ 
				\Z_{\C}\xtworightarrow{\quad} \F_{\C} &,  & \F_{\C}^{\pm}(n)\equiv
				S^{2n-1}\sqcup \left\{ 0 \right\} 
				\end{array}
				\label{eq:3.5}
\end{equation}
We now have a complete picture of the arithmetical bios 
\begin{align}
				\begin{array}[H]{l}
				 \begin{tikzpicture}
								\node at (50mm,35mm) {$\biop_{\C}\;\supseteq\;\Z_{\C}\;\xtworightarrow{\qquad}\F_{\C}$};
								\node at (36mm,30mm) {\rotatebox{90}{\large$\subseteq$}};
								\node at (47mm,30mm) {\rotatebox{90}{\large$\subseteq$}};
								\node at (64mm,30mm) {\rotatebox{90}{\large$\subseteq$}};
								\node at (36mm,20mm) {\rotatebox{90}{\large$\subseteq$}};
								\node at (50mm,25mm) {$\biop_{\R}\;\supseteq\;\Z_{\R}\;\xtworightarrow{\qquad}\F_{\R}$};
								\node at (31mm,15mm) {$\biop_{\Z}\;\subseteq\; \biop_{\Q}$};
								\node at (22mm,05mm) {$\biop_{\F_{p}} \xtwoleftarrow{\qquad}\; \biop_{\Z_p}\,\subseteq\; \biop_{\Q_p}$};
								\node at (24.5mm,10mm) {\rotatebox{270}{\large$\subseteq$}};
								\node at (36mm,10mm) {\rotatebox{270}{\large$\subseteq$}};
				\end{tikzpicture}
				\end{array}
				\label{eq:3.6}
\end{align}
\section{Commutativity} \label{sec:4}
We say the bio $\biop=(\biop^{-},\biop^{+})$ is \myemph{commutative} if the actions and
compositions \myemph{interchange}:
for $\tilde{p}\in\biop^{-}(\ell)$, $q\in \biop^{+}(m)$,
$p\in \biop^{-}(n)$, 
\begin{equation}
				\left( \tilde{p} \overleftarrow{\circ}_{\hspace{-1mm}j}\,q 
				\right){\circ}_{j} p\equiv ( \tilde{p} \circ
								\underbrace{(p,\cdots, p)}_{m})
								\sigma_{m,n}\circleftarrow \underbrace{\left( q,\cdots ,
								q \right)}_{n}
								\label{eq:4.1}
\end{equation}
and similarly,
\begin{equation}
				q \circ_{j} \left( p \overrightarrow{\circ}_{\hspace{-1mm}j}\, \tilde{q} \right) \equiv
				(\underbrace{p,\cdots , p}_{n}) \circrightarrow
				\sigma_{n,m} ( (\underbrace{q,\cdots ,
				q}_{m})\circ \tilde{q})
				\label{eq:4.2}
\end{equation}
for $\tilde{q}\in \biop^{+}(\ell)$, $p\in \biop^{-}(m)$, $q\in
\biop^{+}(n)$. \\
Pictorially, 
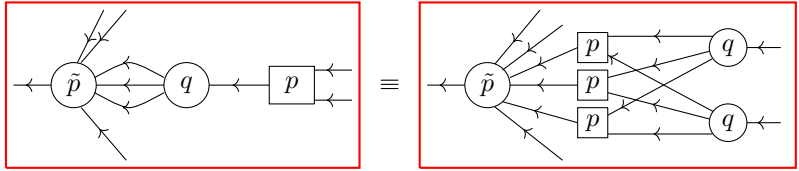
\begin{equation*}
				\begin{array}[H]{l}
				\begin{tikzpicture}[baseline]
								  \def\mline#1#2#3{
												\draw[decorate,decoration={markings,mark=at position #3 with {\arrow[color=black]{<}}}] #1--#2;
												\draw  #1--#2;
							   };
								 \mline{(0mm,20mm)}{(5mm,20mm)}{0.5};
								 \draw (8mm,20mm) circle [radius=3mm];
								 \node at (8mm,20mm) {$\tilde{p}$};
								 \mline{(9mm,23mm)}{(15mm,30mm)}{0.5};
								 \mline{(8.5mm,23mm)}{(12mm,30mm)}{0.5};
								 \mline{(11mm,20mm)}{(20mm,20mm)}{0.5};
								 \mline{(9mm,17mm)}{(15mm,10mm)}{0.5};
								 \draw[-] plot[smooth ] coordinates {(10.8mm,21mm)(15.4mm,23mm)(20mm,21mm) };
								 \draw[decorate,decoration={markings,mark=at position 0.5 with {\arrow[color=black] {<}}}] plot[smooth ] coordinates { (10.8mm,21mm)(15.4mm,23mm)(20mm,21mm)};
								 \draw[decorate,decoration={markings,mark=at position 0.5 with {\arrow[color=black] {<}}}] plot[smooth ] coordinates {(10.8mm,19mm)(15.4mm,17mm)(20mm,19mm) };
								 \draw plot[smooth ] coordinates {(10.8mm,19mm)(15.4mm,17mm)(20mm,19mm) };
								 \draw (23mm,20mm) circle [radius=3mm];
								 \node at (23mm,20mm) {$q$};
								 \draw
								 (40mm,17.5mm)--(40mm,22.5mm)--(34mm,22.5mm)--(34mm,17.5mm)--(40mm,17.5mm); 
								 \node at (37mm,20mm) {$p$};
								 \mline{(26mm,20mm)}{(34mm,20mm)}{0.5};
								 \mline{(40mm,22mm)}{(45mm,22mm)}{0.5};
								 \mline{(40mm,18mm)}{(45mm,18mm)}{0.5};
								 \draw[color=red,thick=3mm] (-1mm,9mm)--(46mm,9mm)--(46mm,31mm)--(-1mm,31mm)--(-1mm,9mm);
								 \node at (50mm,20mm) {$\equiv$};
								 \mline{(55mm,20mm)}{(60mm,20mm)}{0.5};
								 \draw (63mm,20mm) circle [radius=3mm];
								 \node at (63mm,20mm) {$\tilde{p}$};
								 \mline{(64mm,22.8mm)}{(70mm,30mm)}{0.5};
								 \mline{(65mm,22mm)}{(73mm,28mm)}{0.5};
								 \mline{(66mm,21mm)}{(75mm,25mm)}{0.5};
								 \mline{(66mm,20mm)}{(75mm,20mm)}{0.5};
								 \def\stag#1#2{
												 \draw
												 ($#1+(2mm,2mm)$)--($#1+(2mm,-2mm)$)--($#1+(-2mm,-2mm)$)--($#1+(-2mm,2mm)$)--($#1+(2mm,2mm)$);
												 \node at #1 {$#2$};
								 }
								 \stag{(77mm,25mm)}{p};
								 \stag{(77mm,20mm)}{p};
								 \stag{(77mm,15mm)}{p};
								 \mline{(65mm,17.8mm)}{(75mm,15mm)}{0.5};
								 \mline{(64mm,17.1mm)}{(73mm,10mm)}{0.5};
								 \draw (95mm,25mm) circle [radius=2.5mm];
								 \draw (95mm,15mm) circle [radius=2.5mm];
								 \node at (95mm,25mm) {$q$};
								 \node at (95mm,15mm) {$q$};
								 \mline{(79mm,26.5mm)}{(93mm,26.5mm)}{0.5};
								 \mline{(79mm,24mm)}{(93mm,16.5mm)}{0.07};
								 \mline{(79mm,21mm)}{(92.5mm,24.5mm)}{0.5};
								 \mline{(79mm,19mm)}{(92.5mm,15.5mm)}{0.5};
								 \mline{(79mm,13.5mm)}{(92.9mm,13.5mm)}{0.5};
								 \mline{(79mm,16mm)}{(92.9mm,23.5mm)}{0.18};
								 \mline{(97.5mm,25mm)}{(102mm,25mm)}{0.5};
								 \mline{(97.5mm,15mm)}{(102mm,15mm)}{0.5};
								 \draw[color=red,thick=3mm]
								 (54mm,9mm)--(104mm,9mm)--(104mm,31mm)--(54mm,31mm)--(54mm,9mm); 
				\end{tikzpicture}
				\end{array}
\end{equation*} 
Here $\sigma_{m,n}\in S_{m\cdot n}$ is the permutation:
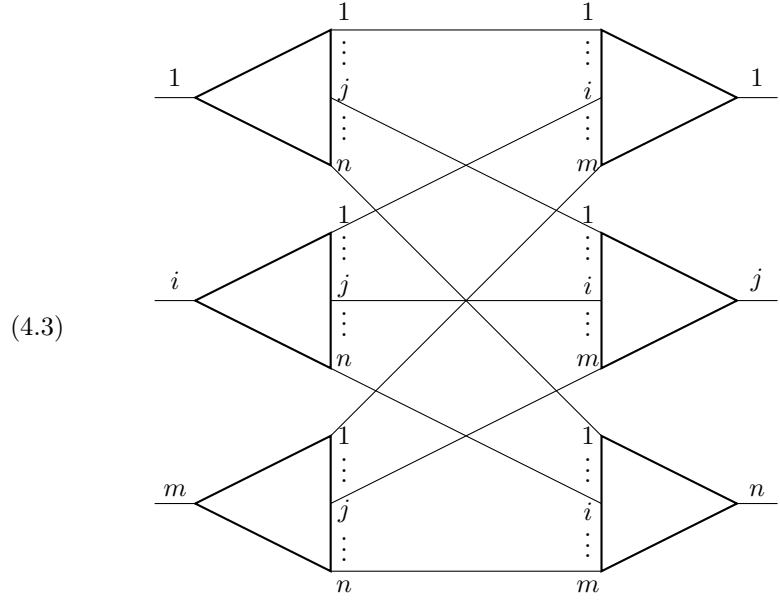
\begin{align}
				\begin{array}[H]{l}
\begin{tikzpicture}[baseline=0mm, scale=1.8]
				 \coordinate (A1) at (15mm,5mm);
				 \coordinate (M1) at (15mm,10mm);
				 \coordinate (B1) at (15mm,15mm);
				 \coordinate (C1) at (5mm,10mm);
				 \coordinate (D1) at (2mm,10mm);
				 \node at (3.5mm,11mm) {$m$};
				 \draw[thick=2] (A1)--(B1)--(C1)--cycle;
				 \draw (C1)--(D1); 
				 \coordinate (A2) at (15mm,20mm);
				 \coordinate (M2) at (15mm,25mm);
				 \coordinate (B2) at (15mm,30mm);
				 \coordinate (C2) at (5mm,25mm);
				 \coordinate (D2) at (2mm,25mm);
				 \node at (3.5mm,26.5mm) {$i$};
				 \draw[thick=2] (A2)--(B2)--(C2)--cycle;
				 \draw (C2)--(D2); 
				 \coordinate (A3) at (15mm,35mm);
				 \coordinate (M3) at (15mm,40mm);
				 \coordinate (B3) at (15mm,45mm);
				 \coordinate (C3) at (5mm,40mm);
				 \coordinate (D3) at (2mm,40mm);
				 \node at (3.5mm,41.5mm) {$1$};
				 \draw[thick=2] (A3)--(B3)--(C3)--cycle;
				 \draw (C3)--(D3); 
				\tikzmath{
								let \x =20mm;
								let \y =2mm;
				};
				 \coordinate (AS1) at (15mm+\x,5mm);
				 \coordinate (MS1) at (15mm+\x,10mm);
				 \coordinate (BS1) at (15mm+\x,15mm);
				 \coordinate (CS1) at (25mm+\x,10mm);
				 \coordinate (DS1) at (28mm+\x,10mm);
				 \node at (26.5mm+\x,11mm) {$n$};
				 \draw[thick=2] (AS1)--(BS1)--(CS1)--cycle;
				 \draw (CS1)--(DS1); 
				 \coordinate (AS2) at (15mm+\x,20mm);
				 \coordinate (MS2) at (15mm+\x,25mm);
				 \coordinate (BS2) at (15mm+\x,30mm);
				 \coordinate (CS2) at (25mm+\x,25mm);
				 \coordinate (DS2) at (28mm+\x,25mm);
				 \node at (26.5mm+\x,26.5mm) {$j$};
				 \draw[thick=2] (AS2)--(BS2)--(CS2)--cycle;
				 \draw (CS2)--(DS2); 
				 \coordinate (AS3) at (15mm+\x,35mm);
				 \coordinate (MS3) at (15mm+\x,40mm);
				 \coordinate (BS3) at (15mm+\x,45mm);
				 \coordinate (CS3) at (25mm+\x,40mm);
				 \coordinate (DS3) at (28mm+\x,40mm);
				 \node at (26.5mm+\x,41.5mm) {$1$};
				 \draw[thick=2] (AS3)--(BS3)--(CS3)--cycle;
				 \draw (CS3)--(DS3); 
				 \draw (B3)--(BS3);
				 \draw (M3)--(BS2);
				 \draw (A3)--(BS1);
				 \draw (B2)--(MS3);
				 \draw (M2)--(MS2);
				 \draw (A2)--(MS1);
				 \draw (B1)--(AS3);
				 \draw (M1)--(AS2);
				 \draw (A1)--(AS1);
				 \node[yshift=-2mm,xshift=1.8mm] at (A1) {$n$};
				 \node[yshift=4mm,xshift=1.8mm] at (A1) {$\vdots$};
				 \node[yshift=-1mm,xshift=1.8mm] at (M1) {$j$};
				 \node[yshift=5.5mm,xshift=1.8mm] at (M1) {$\vdots$};
				 \node[yshift=0mm,xshift=1.8mm] at (B1) {$1$};
				 \node[yshift=1mm,xshift=1.8mm] at (A2) {$n$};
				 \node[yshift=7mm,xshift=1.8mm] at (A2) {$\vdots$};
				 \node[yshift=2mm,xshift=1.8mm] at (M2) {$j$};
				 \node[yshift=8mm,xshift=1.8mm] at (M2) {$\vdots$};
				 \node[yshift=2.5mm,xshift=1.8mm] at (B2) {$1$};
				 \node[yshift=0mm,xshift=1.8mm] at (A3) {$n$};
				 \node[yshift=6mm,xshift=1.8mm] at (A3) {$\vdots$};
				 \node[yshift=1mm,xshift=1.8mm] at (M3) {$j$};
				 \node[yshift=7mm,xshift=1.8mm] at (M3) {$\vdots$};
				 \node[yshift=2.5mm,xshift=1.8mm] at (B3) {$1$};
				 \node[yshift=-2mm,xshift=-1.8mm] at (AS1) {$m$};
				 \node[yshift=4.5mm,xshift=-1.8mm] at (AS1) {$\vdots$};
				 \node[yshift=-1mm,xshift=-1.8mm] at (MS1) {$i$};
				 \node[yshift=5.5mm,xshift=-1.8mm] at (MS1) {$\vdots$};
				 \node[yshift=0mm,xshift=-1.8mm] at (BS1) {$1$};
				 \node[yshift=1mm,xshift=-1.8mm] at (AS2) {$m$};
				 \node[yshift=7mm,xshift=-1.8mm] at (AS2) {$\vdots$};
				 \node[yshift=2mm,xshift=-1.8mm] at (MS2) {$i$};
				 \node[yshift=8mm,xshift=-1.8mm] at (MS2) {$\vdots$};
				 \node[yshift=2.5mm,xshift=-1.8mm] at (BS2) {$1$};
				 \node[yshift=0mm,xshift=-1.8mm] at (AS3) {$m$};
				 \node[yshift=6mm,xshift=-1.8mm] at (AS3) {$\vdots$};
				 \node[yshift=1mm,xshift=-1.8mm] at (MS3) {$i$};
				 \node[yshift=7mm,xshift=-1.8mm] at (MS3) {$\vdots$};
				 \node[yshift=2.5mm,xshift=-1.8mm] at (BS3) {$1$};
         \node at (25mm,0mm) {$(i-1)\cdot n+j = \sigma_{m,n} \big(\left(j-1  \right)m+i\big)$};
\end{tikzpicture}
				\end{array}
				\label{eq:4.3}
\end{align}

\noindent We let $\text{CBio}\subseteq \text{Bio}$ denote the full subcategory of
commutative bios. \vspace{.1cm}\\
For a commutative bio $\biop$ the underlying (associative, unital) monoid $\biop(1)$ is
commutative. Moreover, $\biop(1)$ acts \myemph{centrally} on $\biop^{\pm}(n)$: for $a\in
\biop(1)$
\begin{equation}
				a\circ p = p \circ(a,\cdots , a)\;\;,\;\; p\in\biop^{-}(n)\;\; ; \;\; q\circ a =
				(a,\cdots a)\circ q\;\; , \;\; q\in \biop^{+}(n)
				\label{eq:4.4}
\end{equation}
For a \myemph{multiplicative subset} $S\subseteq\biop(1)$ 
\begin{equation}
				s_{1},s_{2}\in S \Longrightarrow  s_{1}\circ s_{2}\in S,\;\; \unit\in S, 
				\label{eq:4.5}
\end{equation}
we have the \myemph{localization} of $\biop$ at $S$, $S^{-1}\biop$:
\begin{equation}
				\begin{array}[H]{l}
				S^{-1}\biop \equiv \biop[S^{-1}]\equiv (S^{-1}\biop^{-},S^{-1}\biop^{+}) \\\\
				S^{-1}\biop^{\pm}= \biop^{\pm}\times S /\sim 
				\end{array}
				\label{eq:4.6}
\end{equation}
with $(P^{\prime\prime},s^{\prime\prime})\sim (P^{\prime},s^{\prime})$ iff $s\circ
s^{\prime}\circ P^{\prime\prime} = s\circ
s^{\prime\prime}\circ P^{\prime}$ for some $s\in S$; one writes
$P/s=\frac{1}{s}\circ P$ for the equivalence class $(P,s)/\sim$. \vspace{.2cm}\\
For a commutative rig
$A$ the bio $\biop_{A}$ is commutative; this is true in particular for commutative
rings $A$, and we shall continue and define ideals, primes, spectrum, schemes etc. for
commutative bios making sure that we get the right definitions for
$\biop=\biop_{A}$, $A\in C\text{Ring}$.
\begin{remark}
				We denote by $\ctbio\subseteq \cbio$ the full sub-categorys consisting of the
				``totally-commutative'' bios $\biop=(\biop^{-},\biop^{+})$, which, in addition to
				commutativity, satisfy the Boardman-Vogt interchange: For $p\in\biop^{-}(n)$,
				$p^{\prime}\in\biop^{-}(m)$, 
				\begin{equation*}
								p\circ (\underbrace{p^{\prime},\cdots
								,p^{\prime}}_{n})=p^{\prime}\circ(\underbrace{p,\cdots p}_{m})\circ \sigma_{m,n}
				\end{equation*}
				and symmetrically, for $q\in\biop^{+}(n)$, $q^{\prime}\in\biop^{+}(m)$,
				\begin{equation*}
								 (\underbrace{q^{\prime},\cdots
								,q^{\prime}}_{n})\circ q=\sigma_{n,m}\circ(\underbrace{q,\cdots q}_{m})\circ
								q^{\prime} 
				\end{equation*}
				\label{remark:4.7}
\end{remark}
We have the full subcategories
\begin{equation}
				\cring\subseteq \crig \subseteq \ctbio \subseteq \cbio \subseteq \bio
				\label{eq:4.8}
\end{equation}
They are all complete and co-complete. All limits are created in Sets. Also directed co-limits
are created in Sets. These categories also have all co-limits, and in particular we have
push-outs for a diagram 
\begin{equation*}
				\begin{tikzpicture}[baseline=0mm, scale=1.4]
				 \coordinate (A1) at (5mm,5mm) ;
				 \coordinate (B0) at (5mm,13mm) ;
				 \coordinate (B1) at (7mm,15mm) ;
				 \coordinate (A0) at (15mm,15mm) ;
				 \draw[->] (B0)--(A1);
				 \draw[->] (B1)--(A0);
				 \node at (5mm,15mm) {$B$};
				 \node at (5mm,3mm) {$A_1$};
				 \node at (17mm,15mm) {$A_0$};
				\end{tikzpicture}
\end{equation*}
which will be denoted respectively by 
\begin{equation}
		A_0 \underset{B}{\otimes} A_1, \quad A_0 \underset{B}{\otimes} A_1, \quad A_0\underset{B}{\boxtimes}^{T}A_{1},\quad
	A_0\underset{B}{\boxtimes}A_1,\quad A_{0}\underset{B}{\amalg}A_1\quad 
	\label{eq:4.9}
\end{equation}
for the categories of (\ref{eq:4.8}). We have
\begin{equation}
A_0\boxtimes_BA_1=\left(A_0\underset{B}{\amalg}A_1\right)^C\quad , \quad
A_0\underset{B}{\boxtimes}^TA_1 = \left(A_0\underset{B}{\boxtimes}A_1\right)^T,
  \label{eq:4.10}
\end{equation}
where
\begin{equation}
				A\xtworightarrow{\qquad} A^{C}\quad , \text{resp.}\;\;
				A\xtworightarrow{\qquad}A^{T}, 
  \label{eq:4.11}
\end{equation}
is the maximal (resp. totally) commutative quotient of $A$ giving the left adjoint of the
embedding $\cbio\subseteq \bio$, resp. $\ctbio\subseteq \cbio$. 
\begin{remark}
				For (the initial object) $\Z\in C\text{Ring}\subseteq C\text{Rig}$, the
				``arithmetical surface''
				\begin{equation*}
          \text{spec}\left(\biop_{\Z}\boxtimes_{\F} \biop_{\Z}\right) \equiv
          \text{spec}\left( \biop_{\Z} \right)\prod\limits_{\text{spec}(\F)} \text{spec}\left( \biop_{\Z} \right)
				\end{equation*}
				does not reduce to its diagonal, as it does in classical algebraic geometry where 
				\begin{equation*}
								\text{spec}(\Z)\prod \text{spec}(\Z) = \text{spec}(\Z\otimes\Z) =
								\text{spec}(\Z) = \left\{ (0); (2),(3),(5),(7)\cdots \right\}.
				\end{equation*}
				\label{remark:7.9}
\end{remark}
Taking $B=\F$,
$A_{i}=\Z$, in (\ref{eq:4.9}) we get
\[
  \Z \smallcoprod{\F}\Z\xtworightarrow{\;\;\;}
  \Z\boxtimes_{\F}\Z\xtworightarrow{\;\;\;}\Z\boxtimes_{\F}^{T}\Z
  \xrightarrow{\;\overset{\nabla}{\sim}\;}\Z.
\]
If we require total-commutativity the arithmetical surface again reduces to its
diagonal, and $\Z\boxtimes_{\F}^T\Z=\Z$. Indeed, $\Z$ is generated over
$\F$ by $v=(1,1)$, $v^{t}= \begin{pmatrix}1\\ 1
\end{pmatrix} $, (and $(-1)$), and total commutativity implies 
\begin{equation}
        \begin{array}[H]{ll}
          (v\boxtimes 1)\,\circ \hspace{-2.5mm}&(1\boxtimes v, 1\boxtimes v)
          \circ \big((1,0),(0,1) \big) =
          (1\boxtimes v)\circ (v\boxtimes 1 , v\boxtimes 1 )\circ \big( (1,0),(0,1)\big) \\\\
          &\implies  v\boxtimes 1 = 1 \boxtimes v
	\end{array}
	\label{eq:4.12}
\end{equation}
pictorially
\begin{equation}
				\begin{aligned}
				\begin{tikzpicture}[baseline=0mm]
				\draw (2mm,10mm)--(5mm,10mm);
				\draw (5mm,10mm)--(10mm,15mm);
				\draw (5mm,10mm)--(10mm,5mm);
				\draw[dotted] (10mm,15mm)--(17mm,18mm);
				\draw[dotted] (10mm,5mm)--(17mm,2mm);
				\draw[dotted] (10mm,5mm)--(17mm,8mm);
				\draw[dotted] (10mm,15mm)--(17mm,13mm);
                                \node at (19mm,12.5mm)  {$0$};
                                \node at (19mm,17.5mm)  {$1$};
                                \node at (19mm,7.5mm)  {$1$};
                                \node at (19mm,2mm)  {$0$};
                                \node at (44mm,12.5mm)  {$0$};
                                \node at (44mm,17.5mm)  {$1$};
                                \node at (44mm,7.5mm)  {$0$};
                                \node at (44mm,2mm)  {$1$};
				\node at (23mm,10mm) {$\equiv$};
				\draw (17mm,8mm)--(18mm,8mm);
				\draw (17mm,2mm)--(18mm,2mm);
				\draw (17mm,13mm)--(18mm,13mm);
				\draw (17mm,18mm)--(18mm,18mm);

				\draw (27mm,10mm)--(30mm,10mm);
				\draw[dotted] (30mm,10mm)--(35mm,15mm);
				\draw[dotted] (30mm,10mm)--(35mm,5mm);
				\draw (35mm,15mm)--(42mm,18mm);
				\draw (35mm,5mm)--(42mm,2mm);
				\draw (35mm,5mm)--(42mm,8mm);
				\draw (35mm,15mm)--(42mm,13mm);
				\draw (42mm,8mm)--(43mm,8mm);
				\draw (42mm,2mm)--(43mm,2mm);
				\draw (42mm,13mm)--(43mm,13mm);
				\draw (42mm,18mm)--(43mm,18mm);
				\node at (48mm,10mm) {$\Longrightarrow$};
				
				\draw (53mm,10mm)--(55mm,10mm);
				\draw (55mm,10mm)--(60mm,15mm);
				\draw (55mm,10mm)--(60mm,5mm);

				\draw[color=blue,opacity=0.1] (44mm,0mm)--(44mm,20mm)--(0mm,20mm)--(0mm,0mm)-- cycle;
				\draw[color=blue,opacity=0.1] (80mm,0mm)--(80mm,20mm)--(52mm,20mm)--(52mm,0mm)-- cycle;
				\node at (64mm,10mm) {$\equiv$};
				\draw (68mm,10mm)--(70mm,10mm);
				\draw[dotted] (70mm,10mm)--(75mm,15mm);
				\draw[dotted] (70mm,10mm)--(75mm,5mm);
\end{tikzpicture}
				\end{aligned}
				\label{eq:4.13}
\end{equation}
\section{The spectrum} \label{sec:5}
We have the categories of ``\myemph{bios with zero}''
\begin{equation}
				\phantom{*}_\F \backslash\cbio \subseteq\cbio(\text{Set}_{0})\subseteq\cbio_{0}\subseteq\cbio
				\label{eq:5.1}
\end{equation}
where $\biop\in \cbio_0$ iff there exist $0\in\biop(1)$ with $0\circ x=0$ for all
$x\in \biop(1)$; such an element $0$ is unique; the maps $\varphi$ in $\cbio_0$ are required
to preserve it: $\varphi(0)=0$. 
Fix $\biop=(\biop^{-},\biop^{+})\in \text{CBio}_{0}$. 
\begin{definition}
				An \myemph{ideal} is a subset $\oneideal\subseteq\biop(1)$ such that for
				$a_{1}\cdots a_{n}\in \oneideal$, $b\in \biop^{-}(n)$, $d\in\biop^{+}(n)$, the
				\myemph{``linear combination''} 
				\begin{equation}
								\left\{ b,(a_{i}), d  \right\}\overset{\text{def}}{=} \left( b\circ
												(a_{i})
								\right)\circleftarrow d = b\circrightarrow \left( (a_{i})\circ d \right)\in
				\oneideal \;\;\; \text{is in $\oneideal $}
								\label{eq:5.3}
				\end{equation}
				\label{def:5.2}
\end{definition} \vspace{-5mm}
\noindent Note that intersection of ideals is again an ideal, so we can speak of the ideal $\oneideal$
generated by a set $\left\{ a_{i} \right\}_{i\in I}\subseteq\biop(1);\oneideal$ is also given as
the collection of all linear - combinations of the $a_{i}$'s (with repetition). \\
A proper ideal $\pid\subseteq\biop(1)$, with $1\not\in \pid$, is called \myemph{prime} if
$S_{\pid}=\biop(1)\setminus \pid$ is multiplicative. The set of primes is denoted
$\text{spec}(\biop)$.  \\
The set $\text{spec}(\biop)$ is not empty: a maximal proper ideal $1\not\in
\pid\subseteq\biop(1)$, which exists by Zorn's lemma, is always prime. Indeed, if
$a,a^{\prime}\in \biop(1)\setminus \pid$, then for some linear combinations 
\begin{equation*}
				\left\{ b,(a_{i})_{i\le n}, d \right\}=\unit = \left\{
				b^{\prime},(a^{\prime}_{j})_{j\le m},d^{\prime} \right\},b,b^{\prime}\in
				\biop^{-}, d,d^{\prime}\in\biop^{+}, 
\end{equation*}
$a_{i}\in \pid$ or $a_{i}=a$, $a_{j}^{\prime}\in \pid$ or 
$a_{j}^{\prime}=a^{\prime}$ and by commutativity we have
\begin{equation*}
				\begin{array}[H]{lll}
								1 = 1\circ 1 &=& \left(\left(b\circ (a_{i}) \right)\circleftarrow
								d\right)\circ \left(\left(
												b^{\prime}\circ (a_{j}^{\prime})
				\right)\circleftarrow d^{\prime}\right) \\\\
								&=& \left( b\circ (\underbrace{b^{\prime},\cdots b^{\prime}}_{n})
								\right)\sigma \circ
								(a_{i}\circ a_{j}^{\prime})\circleftarrow \sigma^{\prime}\left(
												(\underbrace{d,\cdots , d}_{m})\circ d^{\prime}
								\right)
				\end{array}
\end{equation*}
with $a_{i}\circ a_{j}^{\prime}\in \pid$ or $a_{i}\circ a_{j}^{\prime}= a\circ a^{\prime}$,
and so $a\circ a^{\prime}\not\in \pid$. \\
Moreover, $\text{spec}(\biop)$ is a (compact,
sober$\equiv$Zariski) topological space with respect to the \myemph{Zariski topology} with
closed sets
\begin{equation}
				\begin{array}[H]{l}
				V(\oneideal):= \left\{ \pid\in \text{spec}(\biop),\pid\supseteq \oneideal \right\} \quad ,
				\quad \oneideal\subseteq \biop(1)\; \text{ideal}, \\\\
				V(\oneideal)\cap V(\oneideal^{\prime})=V(\oneideal\cdot \oneideal^{\prime})\quad ,
				\quad \oneideal\cdot\oneideal^{\prime}=\text{the ideal generated by $\left\{
												a\cdot a^{\prime},a\in \oneideal, a^{\prime}\in
												\oneideal^{\prime}
				\right\}$}; \\\\
				\bigcup\limits_{i}V(\oneideal_{i})=V(\sum\limits_{i}\oneideal_{i}) \quad , \quad
				\sum\limits_{i}\oneideal_{i}\;\;\text{the ideal generated by
				$\cup_{i} \oneideal_{i}$}; \\\\
				V(0)=\text{spec}(\biop)\quad , \quad V(1)=\phi
				\end{array}
				\label{eq:5.4}
\end{equation}
We have a basis for the topology by \myemph{basic open sets}
\begin{equation}
				\begin{array}[H]{c}
				D(f):=\left\{ \pid\in \text{spec}(\biop), \pid\not\ni f \right\}\quad , \quad
				f\in\biop(1). \\\\
				D(f_1)\cap D(f_2)=D\left(f_1\circ f_2\right) \quad , \quad D(1)=\text{spec}(\biop)\quad ,
				\quad D(0)=\phi, \\\\
				\text{spec}(\biop)\setminus V(\oneideal)=\bigcup_{f\in \mathfrak{a}}D(f)
				\end{array}
				\label{eq:5.5}
\end{equation}
For a map of commutative bios $\varphi\in \text{CBio}(\biop,\bioq)$, the pull back of a
(prime) ideal is a (prime) ideal, and we obtain
\begin{equation}
				\begin{array}[H]{c}
				\text{spec}(\varphi)=\varphi^{\ast}: \text{spec}(\bioq)\longrightarrow
				\text{spec}(\biop) \\\\
				\varphi^{\ast}(\mathfrak{q})=\varphi_{1}^{-1}(\mathfrak{q})=\left\{
				a\in\biop(1),\varphi(a)\in \mathfrak{q} \right\}.
				\end{array}
				\label{eq:5.6}
\end{equation}
It is a continuous map: 
\begin{equation}
				\varphi^{\ast^{-1}}\big(V(\oneideal)\big) = \left\{ \mathfrak{q}\in
				\text{spec}(\bioq),\varphi^{-1}(\mathfrak{q})\supseteq \oneideal \right\} =
				V\big(\varphi(\oneideal)\big)
				\label{eq:5.7}
\end{equation}
and similarly the pull back of basic open set is again basic open 
\begin{equation}
				\varphi^{\ast^{-1}}\left( D(f) \right) = \left\{ \mathfrak{q}\in \text{spec}(\bioq),
				\varphi^{-1}(\mathfrak{q})\not\ni f \right\} = D\big(\varphi(f)\big). 
				\label{eq:5.8}
\end{equation}
Thus we have a functor 
\begin{equation}
				\text{spec}: (\text{CBio}_0)^{\text{op}} \longrightarrow \top \;
				\text{(compact, sober)}
				\label{eq:5.9}
\end{equation}
We have a Galois correspondence ($\equiv$ adjunction of order sets)
\begin{equation}
				\begin{tikzpicture}[baseline=5mm]
								 \node at (0mm,5mm) {$\left\{ \oneideal \subseteq \biop(1)\;\text{ideal} \right\}$};
								 \draw[->] plot[smooth ] coordinates {(15mm,6mm)(25mm,8mm)(35mm,6mm)};
								 \draw[<-] plot[smooth ] coordinates {(15mm,4mm)(25mm,2mm)(35mm,4mm)};
								 \node at (25mm,10mm) {$V$};
								 \node at (25mm,0mm) {$I$};
								 \node at (53mm,5mm) {$\left\{ Z\subseteq \text{spec}(\biop)\;\text{closed} \right\}$};
				\end{tikzpicture}
				\label{eq:5.10}
\end{equation}
\begin{equation*}
				\bigcap\limits_{\pid\in Z} \pid =: I(Z) \; \tikz \draw[<-|] (0mm,0mm)--(20mm,0mm);\; Z
\end{equation*}
We have $VI(Z)=Z$, and 
\begin{equation}
				IV(\oneideal)=\bigcap\limits_{\mathfrak{p}\supseteq \mathfrak{a}} \pid = \left\{ a\in \biop(1),
								a^{n}=\underbrace{a\circ\ldots \circ a }_{n} \in \oneideal\;\text{for $n>>0$}
\right\}=\sqrt{\oneideal}
\label{eq:5.11}
\end{equation}
We get induced bijections
\begin{equation}
				\begin{tikzpicture}[baseline=5mm]
								 \node at (35mm,6mm) {$\sim$};
								 \node at (10mm,5mm) {$\text{spec}(\biop)$}; 
								 \node at (80mm,5mm) {$\left\{ Z\subseteq\text{spec}(\biop)\;\text{closed irreducible} \right\}$}; 
								 \node at (10mm,15mm) {\rotatebox{90}{\Huge $\subseteq$}};
								 \node at (75mm,15mm) {\rotatebox{90}{\Huge $\subseteq$}};
								 \node at (80mm,23mm) {$\left\{ Z\subseteq\text{spec}(\biop)\;\text{closed } \right\}$}; 
								 \node at (10mm,23mm) {$\left\{ \oneideal\subseteq\biop(1),\oneideal=\sqrt{\oneideal} \right\}$}; 
								 \draw[->] (30mm,24mm)--(60mm,24mm);
								 \draw[arrows={<-}] (30mm,21mm)--(60mm,21mm);
								 \draw[arrows={<->}] (20mm,5mm)--(50mm,5mm);
								 \node at (45mm,26mm) {$V$};
								 \node at (45mm,22mm) {$\sim$};
								 \node at (45mm,19mm) {$I$};
				\end{tikzpicture}
				\label{eq:5.12}
\end{equation}
For a multiplicative set $S\subseteq \biop(1)$, let $\phi_{S}:\biop\rightarrow
S^{-1}\biop$, $\phi_{S}(P)=P/1$, be the canonical map, it induces a homeomorphism 
\begin{equation}
				\begin{tikzpicture}[baseline=5mm]
								 \node at (0mm,5mm) {$\left\{ \pid\in \text{spec}(\biop), \pid\cap S=\phi \right\}$};
								 \draw[<-] plot[smooth ] coordinates {(20mm,6mm)(32.5mm,8mm)(45mm,6mm)};
								 \draw[->] plot[smooth ] coordinates {(20mm,4mm)(32.5mm,2mm)(45mm,4mm)};
								 \node at (32mm,11mm) {$\phi_{S}^{\ast}$};
								 \node at (32mm,0mm) {$S^{-1}$};
								 \node at (32mm,4mm) {$\sim$};
								 \node at (56mm,5mm) {$\text{spec}(S^{-1}\biop)$};
				\end{tikzpicture}
				\label{eq:5.13}
\end{equation}
In particular, for $S_{f}=\left\{ f^{\N} \right\}=\left\{ f^{n},n\ge 0 \right\}$,
$f\in\biop(1)$, we get 
\begin{equation}
				\text{spec}(\biop)\supseteq
				D(f)\xleftrightarrow[\qquad]{\sim}\text{spec}(\biop_{f}), \qquad
				\biop_{f}:= \left\{ f^{\N} \right\}^{-1}\biop
				\label{eq:5.14}
\end{equation}
For $S_{\pid}=\biop(1)\setminus \pid$, $\pid$ prime, we get 
\begin{equation}
				\left\{ q\in\text{spec}(\biop),q\subseteq \pid \right\}
				\xleftrightarrow[\qquad]{\sim}\text{spec}(\biop_{\pid}),\;\biop_{\pid}:=S_{\pid}^{-1}\biop
				\label{eq:5.15}
\end{equation}
i.e. $\biop_{\pid}\in \text{CBio}_{\text{loc}}$ is a \myemph{local} bio in the sense that
\begin{equation}
				m_{\pid}:=\biop_{\pid}(1)\setminus \left\{ x\in \biop_{\pid}(1),\exists x^{-1}\;\text{with}\;
				x\circ x^{-1}=1 \right\}=S_{\pid}^{-1}(\pid)
				\label{eq:5.16}
\end{equation}
is the unique maximal ideal of $\biop_{\pid}$. \\
We make $\text{CBio}_{\text{loc}}$ into a category by 
\begin{equation}
				\text{CBio}_{\text{loc}}(\biop,\bioq)=\left\{ \varphi\in
				\text{CBio}_0(\biop,\bioq),\; \varphi\;\text{is local:
$\varphi^{\ast}(m_{\bioq})=m_{\biop}$} \right\}
\label{eq:5.17}
\end{equation}
\section{The structure sheaf $\mathscr{O}_{\biop}$} \label{sec:6}
Fix $\biop\in \text{CBio}_0$. For an open set $\mathscr{U}\subseteq \text{spec}(\biop)$ define 
\begin{equation}
  \mathscr{O}^{\pm}_{\biop}(\mathscr{U})(n):=\left\{ f:\mathscr{U}\to\coprod\limits_{\pid\in
								\mathscr{U}}\biop^{\pm}_{\mathfrak{p}}(n),
								f(\mathfrak{p})\in\biop^{\pm}_{\mathfrak{p}}(n), \;\text{and $f$ is a
\myemph{local fraction}} \right\}	
\label{eq:6.1}
\end{equation}
$f$ \myemph{locally a fraction}: for all $\mathfrak{p}\in \mathscr{U}$, there exists open
$\mathfrak{p}\in
\mathscr{U}_{\mathfrak{p}}\subseteq \mathscr{U}$, and $P\in \biop^{\pm}(n)$, $s\in \biop(1)\setminus
\bigcup\limits_{\mathfrak{q}\in \mathscr{U}_{\mathfrak{p}}} \mathfrak{q}$, such that for all
$\mathfrak{q}\in \mathscr{U}_{\mathfrak{p}}$ we have $f(\mathfrak{q})\equiv
P/ s $ in $\biop_{\mathfrak{q}}$. \\
Note that $\mathscr{O}_{\biop}(\mathscr{U})=\left(
\mathscr{O}_{\biop}^{-}(\mathscr{U}),\mathscr{O}_{\biop}^{+}(\mathscr{U}) \right)$ is a
commutative bio, the operations of compositions and actions are defined
pointwise (in each $\biop_{\mathfrak{p}}$), and the
``local fraction condition'' is preserved. For open sets $\mathscr{V}\subseteq \mathscr{U}\subseteq
\text{spec}(\biop)$ we have the restriction maps 
\begin{equation}
				\rho_{\mathscr{V}}^{\mathscr{U}}: \biop(\mathscr{U})\to \biop(\mathscr{V})
				\label{eq:5.19}
\end{equation}
making $\mathscr{U}\mapsto \biop(\mathscr{U})$ a pre-sheaf of bios over $\text{spec}(\biop)$; by the local
nature of the ``locally-fraction-condition'' it is clearly a sheaf. \\
For $\mathfrak{p}\in \text{spec}(\biop)$, the \myemph{stalk} of $\mathscr{O}_{\biop}$ at
$\mathfrak{p}$ is given by 
\begin{equation}
				\begin{array}[H]{c}
								\mathscr{O}_{\biop, \mathfrak{p}}:=
								\limrightarrow\limits_{\mathscr{U}\ni
								\mathfrak{p}}\mathscr{O}_{\biop}(\mathscr{U})\xrightarrow{\sim}\biop_{\mathfrak{p}} \\\\
								(\mathscr{U},f)_{/\approx} \xmapsto{\quad} f(\mathfrak{p})
				\end{array}
				\label{eq:6.3}
\end{equation}
it is well defined, surjective, and injective. \\
The global sections of $\mathscr{O}_{\biop}$ are given by 
\begin{theorem}
				For a basic open set $D(s)\subseteq \text{spec}(\biop)$, $s\in \biop(1)$,
				\begin{equation*}
								\begin{array}[H]{c}
								\Psi:\biop_{s}=\biop\left[ \frac{1}{s} \right]\xrightarrow{\quad\sim\quad}
								\mathscr{O}_{\biop}\left( D(s) \right) \\\\
								\Psi\left( P/s^{n} \right):= \left\{ f(\mathfrak{p})\equiv P/s^{n} \; \text{in
								$\biop_{\mathfrak{p}}$ for all $\mathfrak{p}\in D(s)$} \right\}.
								\end{array}
				\end{equation*}
				\label{thm:1}
\end{theorem}
For $s=1$ we obtain the global sections 
\begin{equation*}
				\biop\xrightarrow{\quad \sim \quad}\mathscr{O}_{\biop}(\text{spec}(\biop))
\end{equation*}
\begin{proof}
				The map $\Psi$ which takes $P/s^{n}\in \biop_{s}$ to the constant section
				$f$, $f(\mathfrak{p})\equiv P/s^{n}$, is a well defined map of bios. Take $P\in
				\biop^{-}(k)$, the case $P\in\biop^{+}(k) $ is similar. \\
				\myemph{$\Psi$ is injective}: Assume $\Psi\left( P_{1}/s^{n_1}
				\right)=\Psi(P_2/s^{n_2})$, and define 
				\begin{equation*}
								\oneideal := \text{Ann}\left( s^{n_2}\circ P_{1} ; s^{n_1}\circ P_2
								\right)=\left\{ a\in \biop(1), a\circ s^{n_2}\circ P_{1}=a\circ
								s^{n_{1}}\circ P_2 \right\}; 
				\end{equation*}
				by commutativity it follows $\oneideal $ is an ideal. We have
				\begin{equation*}
								P_1/s^{n_1}=P_2/s^{n_2}\;\; \text{in}\;\; \biop_{\mathfrak{p}} \;\;\text{for all}\;\; \mathfrak{p}\in D(s)
				\end{equation*}
				\begin{description}
								\item[$\Rightarrow$ ]  $s_{\mathfrak{p}}\circ s^{n_2}\circ
												P_1=s_{\mathfrak{p}}\circ
												s^{n_1}\circ P_2$ for some $s_{\mathfrak{p}}\in \biop(1)\setminus
												\mathfrak{p}$,
												all $\mathfrak{p}\in D(s)$. 
								\item[$\Rightarrow$ ]  $\oneideal\not\subseteq\mathfrak{p}$, all
												$\mathfrak{p}\in
												D(s)$
								\item[$\Rightarrow$ ]  $V(\oneideal)\cap D(s)=\emptyset\Rightarrow
												V(\oneideal)\subseteq V(s)\Rightarrow s\in
												IV(\oneideal)=\sqrt{\oneideal}$ 
								\item[$\Rightarrow$ ] $s^{n}\in \oneideal$ for $n>>0 \Rightarrow
												s^{n+n_{2}}\circ P_{1}=s^{n+n_1}\circ P_{2}$ for $n>>0$
								\item[$\Rightarrow$ ] $P_1/s^{n_1}=P_2/s^{n_2}$ in $\biop_{s}$.
				\end{description}
				\myemph{$\Psi$ is surjective}: Fix $f\in
				\mathscr{O}_{\biop}(D(s))^{-}(k)$. Since $D(s)$ is compact we can cover it by
				finite collection of basic open sets 
				\begin{equation*}
								D(s)=D(s_{1})\cup\cdots\cup D(s_{N})
				\end{equation*}
				with 
				\begin{equation*}
								f(\mathfrak{p})\equiv P_{i}/t_{i} \;\text{for}\; \mathfrak{p} \in D(s_{i})\quad ;\quad
								t_{i}\in \biop(1)\setminus\bigcup\limits_{\mathfrak{p}\in D(s_{i})}\pid.
				\end{equation*}
				We have $V(t_{i})\subseteq V(s_{i})$, so $s_{i}^{n_i}=c_{i}\circ t_{i}$ for some
				$c_{i}\in \biop(1)$, and for $\mathfrak{p}\in D(s_{i})$, $f(\mathfrak{p})=P_{i}/t_{i}=c_{i}\circ
				P_{i}/s_{i}^{n_i}$. We can replace $s_{i}$ by $s_{i}^{n_i}$,
				($ D(s_{i})=D(s_{i}^{n_{i}})$), and $P_{i}$ by $c_{i}\circ P_{i}$, so 
				\begin{equation*}
								f(\mathfrak{p})\equiv P_{i}/s_{i}\quad \text{for}\; \mathfrak{p}\in D(s_{i}).
				\end{equation*}
				On the set $D(s_{i})\cap D(s_{j})=D(s_{i}\circ s_{j})$, the section $f$ is given
				by both $P_{i}/s_{i}$ and $P_{j}/s_{j}$, and by the injectivity  of $\Psi$ 
				\begin{equation*}
								(s_{i}\circ s_{j})^{n}\circ s_{j}\circ P_{i} = (s_{i}\circ
								s_{j})^{n}\circ s_{i}\circ P_{j}
				\end{equation*}
				By finiteness we may assume one $n$ works for all $i,j$. \vspace{.1cm}\\ 
				Replacing $s_{i}$ by
				$s_{i}^{n+1}$, and replacing $s_{i}^{n}\circ P_{i}$ by $P_{i}$, we may assume 
				\begin{equation*}
								\begin{array}[H]{l}
												f(\mathfrak{p})\equiv P_{i}/s_{i} \quad \text{for all
																$\mathfrak{p}\in
												D(s_{i})$} \\\\
												s_{j}\circ P_{i}= s_{i}\circ P_{j} \quad \text{all $i,j$.}
								\end{array}
				\end{equation*}
				Since $D(s)\subseteq \bigcup\limits_{i}D(s_{i})$ we have that some power
				$s^{M}$ is a linear-combination of the $s_{i}$
				\begin{equation*}
								\begin{array}[H]{l}
								s^{M}=\left\{ b,(c_{j}),d \right\} = b\circ (c_{j})\circleftarrow d \quad
								, \quad c_{j}=s_{i(j)}\quad ,\quad b,d\in \biop^{\pm}(\ell)\quad ,\\\\
								j\in \left\{ 1,\cdots , \ell \right\}\xrightarrow{i}\left\{
												1,\cdots, N
								\right\}.
								\end{array}
				\end{equation*}
				Define $P=b\circ (P_{i(j)})\circleftarrow (\underbrace{d,\ldots ,d}_{k})$. \\
				We have, 
				\begin{equation*}
								\begin{array}[H]{lll}
												s_{j}\circ P &=&  b\circ \left( s_{j}\circ P_{i(j)}
												\right)\circleftarrow (d,\cdots , d)\equiv  b\circ \left(
												s_{i(j)}\circ P_{j} \right)\circleftarrow (d,\cdots , d) \\\\
												&=& (b\circ (c_{j}\circ P_{j}))\circleftarrow (d,\cdots , d) \\\\
												&=& ((b\circ(c_{j}))\circleftarrow d)\circ P_{j}\qquad
												\text{by commutativity} \\\\
												&=& s^{M}\circ P_{j}
								\end{array}
				\end{equation*}
				Thus $f(\mathfrak{p})\equiv P_{j}/s_{j}\equiv P/s^{M}$ is constant. 
\end{proof}
\section{Generalized Schemes} \label{sec:7}
For a topological space $\xspace$, we have the category $\text{CBio}/_{\xspace}$ of
sheaves of $\text{CBio}_0$ over $\xspace$, its maps are natural transformations
$\varphi=\left\{ \varphi_{\mathscr{U}} \right\}$, $\varphi_{\mathscr{U}}\in
\text{CBio}_0(\biop(\mathscr{U}),\biop^{\prime}(\mathscr{U}))$. Putting all these categories together we have
the category $\text{CBio}/\top$: its object are pairs $(\xspace,\biop)$,
$\xspace\in \top$, $\biop\in \text{CBio}/\xspace$, and its maps
$f:(\xspace,\biop)\to(\xspace^{\prime},\biop^{\prime})$ are pairs $f\in
\top(\xspace,\xspace^{\prime})$ and $f^{\natural}\in
\text{CBio}/_{\xspace^{\prime}} (\biop^\prime,f_{\ast}\biop) $; explicitly, $f$ is a
continuous function, and for $\mathscr{U}\subseteq \xspace^{\prime}$ open, we have the map of bios
$f_{\mathscr{U}}^{\natural}\in \text{CBio}_0 \left( \biop^{\prime}(\mathscr{U}),\biop\left(
f^{-1}(\mathscr{U}) \right) \right)$,
these maps being compatible with restrictions. 
\begin{remark}
				For $f\in \top(\xspace,\xspace^{\prime})$ we have adjunction
\begin{figure}[H]
				\centering
				\begin{tikzpicture}[baseline=0mm]
								\node at (41.5mm,22mm) {$\text{CBio}/\xspace$};
								 \draw[<-] plot[smooth ] coordinates {(38mm,20mm)(35mm,14mm)(38mm,7mm)};
								 \draw[->] plot[smooth ] coordinates {(45mm,20mm)(48mm,14mm)(45mm,7mm)};
								 \node at (41.5mm,4mm) {$\text{CBio}/\xspace^{\prime}$};
								 \node at (32mm,14mm) {$f^{\ast}$};
								 \node at (51mm,14mm) {$f_{\ast}$};
								 \node at (0mm,12mm) {$f^{\ast}\biop^{\prime} = $ sheaf associated};
								 \node at (75mm,12mm) {$f_{\ast}\biop(\mathscr{U})=\biop(f^{-1}\mathscr{U})$ };
								 \node at (10mm,4mm) {to the pre-sheaf};
								 \node at (35mm,-7mm) {
																			$ \hspace{-6cm}  \xspace \supseteq \mathscr{V} \xmapsto{\qquad} \hspace{-10mm} \limrightarrow\limits_{
								\begin{array}[H]{l}
												\hspace{10mm} \mathscr{U}\subseteq \xspace^{\prime}\,\text{open} \\
												\hspace{10mm} f(\mathscr{V})\subseteq \mathscr{U}
								\end{array}
				} 
				\hspace{-7mm}\biop^{\prime}(\mathscr{U})
								 $};
				\end{tikzpicture}
\end{figure}
\noindent For a map $f\in \text{CBio}/{\top} \left(
(\xspace,\biop),(\xspace^{\prime},\biop^{\prime}) \right)$, and for a point $x\in
\xspace$, we get an induced map of stalks 
\begin{equation}
				f_{x}^{\natural}:\biop^{\prime}_{f(x)}= \limrightarrow\limits_{f(x)\in \mathscr{V}\subseteq
				\xspace^{\prime}}
				\biop^{\prime}(\mathscr{V})\xrightarrow{\qquad}\limrightarrow\limits_{f(x)\in
								\mathscr{V}\subseteq
				\xspace^{\prime}} \biop(f^{-1}\mathscr{V})\to \limrightarrow\limits_{x\in
								\mathscr{U}\subseteq
				\xspace}\biop(\mathscr{U})=\biop_{x} 
				\label{eq:7.2}
\end{equation}
\label{remark:7.1}
\end{remark}
\begin{definition}
				The category of \myemph{locally-bio-spaces}
				$\text{CBio}_{\text{loc}}/_{\top}$, is the category with object
				$(\xspace,\biop)\in \text{CBio}/_{\top}$, such that for all $x\in \xspace$
        the stalk $\biop_{x}$ is a \myemph{local} bio with the unique maximal proper ideal
				$m_{x}\subseteq\biop_{x}(1)$; the maps 
				$f\in \text{CBio}_{\text{loc}}/_{\top} \left(
				(\xspace,\biop),(\xspace^{\prime},\biop^{\prime}) \right) $ are maps $f\in
				\text{CBio}/_{\top}\left( (\xspace,\biop),(\xspace^{\prime},\biop^{\prime})
				\right)$ such that for all $x\in X$, $f^{\natural}_{x}\in \text{CBio}_{\text{loc}}\left(
				\biop^{\prime}_{f(x)},\biop_{x} \right)$ is a \myemph{local} map:
				$f_{x}^{\natural}(m_{f(x)})\subseteq m_{x}$.
				\label{def:7.3}
\end{definition}
\begin{theorem}
				We have the adjunction
\begin{figure}[H]
				\centering
				\begin{tikzpicture}[baseline=0mm]
								 \node at (41.5mm,23mm) {$\left(\text{CBio}_{\text{loc}}/_{\top}\right)^{\text{op}}$};
								 \draw[<-] plot[smooth ] coordinates {(38mm,20mm)(35mm,14mm)(38mm,7mm)};
								 \draw[->] plot[smooth ] coordinates {(45mm,20mm)(48mm,14mm)(45mm,7mm)};
								 \node at (41.5mm,4mm) {$\text{CBio}_0$};
								 \node at (6mm,14mm) {$\text{spec}(\biop):=\left(
								 \text{spec}(\biop),\mathscr{O}_{\biop} \right)$ \hspace{6mm} spec};
								 \node at (68mm,14mm) {$\Gamma$ \hspace{4mm} $\Gamma(\xspace, \biop):=\biop(\xspace)$};
								 \node at (70mm,9mm) {the global sections};
				\end{tikzpicture}
\end{figure}
\begin{equation*}
				\text{CBio}_{\text{loc}}/_{\top} \left(
				(\xspace,\biop),\text{spec}(\biop^{\prime}) \right)\equiv \text{CBio}_0\left(
								\biop^{\prime},\biop(\xspace)
				\right)
\end{equation*}
\label{thm:7.2}
\end{theorem}
\begin{proof}
				For a point $x\in \xspace$, we have canonical map $\phi_x\in \text{CBio}_0\left(
				\biop(\xspace),\biop_{x} \right)$, $\phi_{x}P=P|_{x}$ the stalk of the global
				section $P$ at the point $x$, and we get a prime
				$\pid_{x}:=\phi_{x}^{-1}(m_{x})\in \text{spec}\left(\biop(\xspace)\right)$, $m_{x}\subseteq
				\biop_{x}(1)$ the maximal ideal. For a basic open set $D(s)\subseteq
				\text{spec}\left(\biop(\xspace)\right)$, $s\in \biop(\xspace)(1)$, we have $\left\{ x\in
								X,\pid_{x}\in D(s)
				\right\}=\left\{ x\in X,\phi_{x}(s)\not\in m_{x} \right\}$ is \myemph{open} in
				$\xspace$: if $\phi_{x}(s)\not\in m_{x}$, it is invertible in $\biop_{x}$,
				$a_{x}\circ \phi_{x}(s)=1$, $a_{x}\in\biop_{x}$ and there is an open
				$\mathscr{U}_{x}\ni x$,
				$a\in \biop(\mathscr{U}_{x})$ with $a_{x}=a|_{x}$; taking $\mathscr{U}_{x}$ smaller we have $a\circ
				s|_{\mathscr{U}_{x}}=1$ already in $\biop(\mathscr{U}_x)$, and so $\phi_{x^{\prime}}(s)\not\in
				m_{x^{\prime}}$ for all $x^{\prime}\in \mathscr{U}_{x}$. Thus the map
				$x\mapsto\pid_{x}$ is a continuous map
				$\pid:\xspace\xmapsto{\quad}\text{spec}\left(\biop(\xspace)\right)$.
				The uniqueness of the inverse $\left( s|_{u_{x}} \right)^{-1}$, shows these
				local inverses glue to a global inverse $s^{-1}\in \biop\left( \pid^{-1}\left(
				D(s) \right) \right)$ and we get the map of $\text{CBio}_0$
				\begin{equation}
								\pid_{D(s)}^{\natural}: \biop(\xspace)_{s}\equiv \left\{ s^{\N}
								\right\}^{-1}\biop(\xspace)\xrightarrow{\qquad}\biop\left(\pid^{-1}\left(
																D(s)
								\right)\right)
								\label{eq:7.4}
				\end{equation}
				These maps are compatible on intersections, $D(s_{1})\cap D(s_{2})=D(s_{1}\circ
				s_{2})$, so by the sheaf property give $\pid_{\mathscr{U}}^{\natural}\in
				\text{CBio}_0\left(
								\mathscr{O}_{\text{spec}\; \biop(\xspace)}(\mathscr{U}),
								\biop(\pid^{-1}\mathscr{U})
				\right)$ for any open $\mathscr{U}\subseteq \text{spec}\left(\biop(\xspace)\right)$, 
				compatible with restrictions, so 
				\begin{equation}
								\pid = (\pid,\pid^{\natural})\in \text{CBio}/_{\top}\left(
								(\xspace,\biop),\text{spec}\left(\biop(\xspace)\right) \right)
								\label{eq:7.5}
				\end{equation}
				For $x\in\xspace$, we get 
				\begin{equation}
								\pid_{x}^{\natural}=\limrightarrow\limits_{\phi_{x}(s)\not\in m_{x}}
								\pid^{\natural}_{D(s)}\in \text{CBio}_{\text{loc}}\left(
								\biop(\xspace)_{\pid_{x}},\biop_{x} \right)
								\label{eq:7.6}
				\end{equation}
				is \myemph{local}, so $\pid\in
				\text{CBio}_{\text{loc}/\top}\left((\xspace,\biop),\text{spec}\left(\biop( \xspace )\right)\right)$
is the co-unit of adjunction. \vspace{.1cm}\\
				Given $\varphi\in \text{CBio}_0\left(
				A,\biop(\xspace) \right)$ we get $\text{spec}(\varphi)\circ\pid\in
				\text{CBio}_{\text{loc}/\top}\left( (\xspace,\biop),\text{spec} (A) \right)$.
				\vspace{.1cm}\\
				Given  $f\in \text{CBio}_{\text{loc}/\top}\left(
				(\xspace,\biop),\text{spec}(A) \right)$ we get the global sections 
				\begin{equation}
								\Gamma(f)=f_{\text{spec}(A)}^{\natural}\in \text{CBio}_0(A,\biop(\xspace)).
								\label{eq:7.7}
				\end{equation}
				One  checks these are inverse bijections.
\end{proof}
\noindent Following the footsteps of Grothendieck
we can define the category of \break
\myemph{Generalized Schemes} $\text{GSch} \subseteq
\text{CBio}_{\text{loc}}/_{\top}$ to be the full subcategory of
				$\text{CBio}_{\text{loc}}/_{\top}$ consisting of the object
				$(\xspace,\mathscr{O}_{\xspace})$ which are locally affine: we have some open
				cover $\xspace=\bigcup
				\mathscr{U}_{i}$, and $(\mathscr{U}_{i},\mathscr{O}_{\xspace}|_{\mathscr{U}_{i}})\cong
				\text{spec}\left(\mathscr{O}_{\xspace}(\mathscr{U}_{i})\right)$. \vspace{.1cm}\\
An open subset of a scheme is again a scheme. \vspace{.1cm}\\
Schemes can be glued along open subsets and consistent glueing data.
				\vspace{.1cm}\\
				Since ordinary commutative rings $A$ give commutative bios $\biop_{A}\in
				\text{CBio}_0$ and since all our definitions reduce to their classical analogues for
				$\biop=\biop_{A}$, $A\in C\text{Ring}$, ordinary schemes embeds fully faithfully in generalized
				schemes, 
				$$\sch\hookrightarrow \text{GSch},
(\xspace,\mathscr{O}_{\xspace})\mapsto (\xspace,\biop_{\mathscr{O}_{\xspace}})$$
\begin{theorem}
				The category $\text{GSch}$ has fiber products, for  $f_{i}\in\text{GSch}(X_{i},Y)$:
\begin{figure}[H]
				\centering
				\begin{tikzpicture}[baseline=0mm]
								 \node at (50mm,30mm) {$X_{0}\prod\limits_{Y}X_{1}$};
								 \node at (35mm,15mm) {$X_{0}$};
								 \node at (65mm,15mm) {$X_{1}$};
								 \node at (50mm,0mm) {$Y$};
								 \draw[->] (54mm,28mm)--(64mm,18mm);
								 \draw[->] (46mm,28mm)--(36mm,18mm);
								 \draw[-] (43mm,23mm)--(50mm,18mm)--(57mm,23mm);
								 \draw[->] (36mm,13mm)--(47mm,2.5mm);
								 \draw[->] (63mm,13mm)--(52mm,2.5mm);
								 \node at (38mm,8mm) {$f_{0}$};
								 \node at (61mm,8mm) {$f_{1}$};

				\end{tikzpicture}
\end{figure}
				\label{thm:3}
\end{theorem}
\begin{proof}
Exactly as for ordinary schemes. Write $Y=\bigcup\limits_{i}\text{spec}(B_{i})$, 
$X_{\varepsilon}=\bigcup\limits_{i,j} \text{spec}(A_{j,i}^{\varepsilon})$, $\varepsilon=0,1$
with $f_{\varepsilon}\left(\text{spec}(A_{j,i}^{\varepsilon})\right)\subseteq
\text{spec}\; (B_{i})$,  and glue:
\begin{equation}
				X_0\prod\limits_{Y} X_{1}\equiv
				\coprod_{i,j_{0},j_{1}}\text{spec}\left(
				A^{0}_{j_{0},i}\underset{B_i}{\boxtimes}A_{j_{1},i}^{1} \right).
								\label{eq:7.8}
\end{equation}
Note that one uses the push-out
$\left(A^{0}_{j_0,i}\coprod\limits_{B_{i}}A_{j_{1},i}^{1}\right)^{C}$ in $\text{CBio}_0$
\end{proof}

\section{Pro-Schemes} \label{sec:8}
We have the full-embeddings 
\begin{equation}
				\text{Aff}\equiv \left( \cbio_{0} \right)^{\op} \subseteq \text{GSch} \subseteq
				\cbio_{\text{loc}}/\top
				\label{eq:8.1}
\end{equation}
The category $\cbio_{\text{loc}}/\top=\mathcal{C}$ is complete. Given 
$\left\{ \xspace_i \right\}\in\mathcal{C}^{I}$ we have the topological space
$\lim\limits_{\underset{I}{\leftarrow}}\xspace_{i}$, its points $x=(x_{i})$ are 
consistent sequences of points $x_{i}\in \xspace_{i}$, and we have the structure
sheaf $\mathscr{O}=\lim\limits_{\underset{I}{\rightarrow}}\mathscr{O}_{\xspace_{i}}$ over it
(with local stalks: $\mathscr{O}$ being the co-limit of
local homomorphisms). The category $\text{Aff}$ is also complete, as $\cbio_0$ is
co-complete
\begin{equation}
				\lim_{\underset{I}{\leftarrow}}\text{spec}(A_i)\equiv\text{spec}(\lim\limits_{\underset{I}{\rightarrow}}A_{i}) 
				\label{eq:8.2}
\end{equation}
But the category $\text{GSch}$, the full sub-category of
$\cbio_{\text{loc}}/\top$ with objects the ones locally isomorphic to elements of
$\text{Aff}$, is not complete: Given $\left\{ \xspace_{i} \right\}\in \left(
				\text{GSch}
\right)^{I}$, and given a point 
\begin{equation}
				x=(x_{i}) \in \lim_{\underset{I}{\leftarrow}}\xspace_i\in \cbio_{\text{loc}}/\top
				\label{eq:8.3}
\end{equation}
while for each $i\in I$ we have an affine neighborhood $x_{i}\in \text{spec}
A_{i}\subseteq\xspace$ there need not be an affine neighborhood of $x$ in
$\lim\limits_{\underset{I}{\leftarrow}}\xspace_{i}$. Such is the case of the real prime. 
So we pass to the pro-category pro-GSchm. It has objects $(J,\xspace_{j})$ where $J$ is a
partially ordered set which is 
\begin{equation}
				\text{directed:}\quad \forall j_1,j_2\in J,\;\; \exists j\in J ,\;\; j\ge j_1 \;\;
				\text{and}\;\; j\ge j_{2}
				\label{eq:8.4}
\end{equation}
\begin{equation}
				\begin{array}[H]{l}
				\text{co-finite:}\quad \forall j\in J,\;\; \#\left\{ i\in J, i\le j
				\right\}<\infty \\\\
\xspace_{j}\in (\text{GSch})^{J}\;:\quad j_{1}\ge j_{0}\;\;\Longrightarrow \;\;
				\xspace_{j_1}\longrightarrow \xspace_{j_0}
				\end{array}
				\label{eq:8.5}
\end{equation}
The maps in pro-$\text{GSch}$ are given by
\begin{equation}
				\text{pro-GSch}\left( (J,\xspace_{j}), (I,Y_{i}) \right)=
				\lim\limits_{\underset{I}{\leftarrow}}\lim\limits_{\underset{J}{\rightarrow}}
				\text{GSch}\left( \xspace_{j},Y_{i} \right)
				\label{eq:8.6}
\end{equation}
We have a functor 
\begin{equation}
				\begin{array}[H]{llll}
								\lim\limits_{\longleftarrow}\;\; : \;\; & \text{pro-GSch}&\longrightarrow&
								\cbio_{\text{loc}}/\top \\\\
								\ & (J,\xspace_{j})&\xmapsto{\quad}&
								\lim\limits_{\overset{\longleftarrow}{J}}\xspace_{j}
				\end{array}
				\label{eq:8.7}
\end{equation}
The category $\text{GSch}$ embeds fully faithfully in pro-$\text{GSch}$; the category
pro-\text{GSch} is closed under limits over directed co-finite posets. 
\begin{example}
				Take 
				\begin{equation}
								J=\left\{ n\in\N \;\text{square free} \right\}=\left\{ p_{1}\cdot
								p_{2}\cdots p_{\ell},p_{i}\;\text{prime},i\not= j\Rightarrow p_{i}\not=
				p_{j} \right\}
				\label{eq:8.9}
\end{equation}
with $n_{1}\ge n_{0}$ iff $n_{0}$ divides $n_1$. We have the ``\myemph{compactified
$\text{spec}(\Z)$}'' 
\begin{equation}
				\overline{\text{spec}(\Z)}:=\left\{
								\xspace_{n}=\text{spec}(\Z)\coprod\limits_{\text{spec}(\Z\left[ \frac{1}{n}
\right])}\text{spec}\left( \Z\left[ \frac{1}{n} \right]\cap\Z_{\R} \right)
\right\}_{n\in J}
\label{eq:8.10}
\end{equation}
Pictorially, the specialization picture of the underlying topological space $\xspace_n$
looks like 
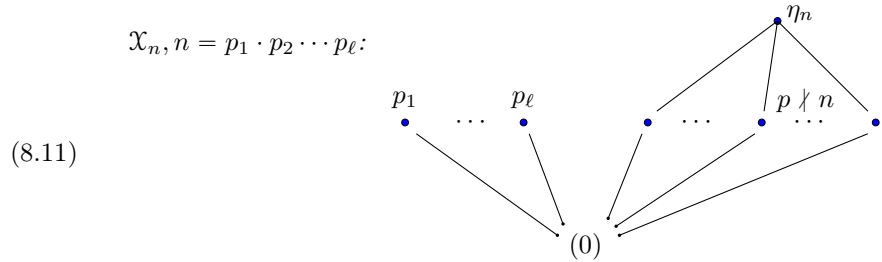
\begin{equation}
				\begin{aligned}
								\begin{tikzpicture}[baseline=0mm,scale=1.5]
				\node at (10mm,23mm) {$\xspace_n, n=p_1\cdot p_2\cdots p_{\ell}$:};
				\coordinate (A) at (40mm,5mm);
				\coordinate (A1) at (43mm,6mm);
				\draw[fill=black] (A1) circle [radius=0.1mm];
				\coordinate (A2) at (42.7mm,6.8mm);
				\draw[fill=black] (A2) circle [radius=0.1mm];
				\coordinate (A3) at (42mm,7.5mm);
				\draw[fill=black] (A3) circle [radius=0.1mm];
				\coordinate (A4) at (38mm,7mm);
				\draw[fill=black] (A4) circle [radius=0.1mm];
				\coordinate (A5) at (37.5mm,6mm);
				\draw[fill=black] (A5) circle [radius=0.1mm];
				\node[yshift=0mm] at (A) {$(0)$} ;
				\coordinate (B5) at (25mm,15mm);
				\draw (B5)--(A5);
				\coordinate (B4) at (35mm,15mm);
				\draw (B4)--(A4);
				\coordinate (B3) at (45mm,15mm);
				\draw (B3)--(A3);
				\coordinate (B2) at (55mm,15mm);
				\draw (B2)--(A2);
				\coordinate (B1) at (65mm,15mm);
				\draw (B1)--(A1);
				\node at (30 mm, 16mm) {$\cdots$};
				\node at (50 mm, 16mm) {$\cdots$};
				\node at (60 mm, 16mm) {$\cdots$};
				\node at (59.5mm, 18mm) {$ p\not \divides n$};
				\coordinate (BB5) at  ($(B5)+(-1mm,1mm)$);
				\draw[fill=blue] (BB5) circle [radius=0.3mm];
				\coordinate (BB4) at  ($(B4)+(-0.5mm,1mm)$);
				\draw[fill=blue] (BB4) circle [radius=0.3mm];
				\coordinate (BB3) at  ($(B3)+(0.5mm,1mm)$);
				\draw[fill=blue] (BB3) circle [radius=0.3mm];
				\coordinate (BB2) at  ($(B2)+(0.6mm,1mm)$);
				\draw[fill=blue] (BB2) circle [radius=0.3mm];
				\coordinate (BB1) at  ($(B1)+(0.7mm,1mm)$);
				\draw[fill=blue] (BB1) circle [radius=0.3mm];
				\coordinate (C) at  ($(B2)+(2mm,10mm)$);
				\draw[fill=blue] (C) circle [radius=0.3mm];
				\coordinate (BBB3) at  ($(BB3)+(0.8mm,1mm)$);
				\draw (BBB3)--(C);
				\coordinate (BBB2) at  ($(BB2)+(0.2mm,1mm)$);
				\draw (BBB2)--(C);
				\coordinate (BBB1) at  ($(BB1)+(-0.7mm,1mm)$);
				\draw (BBB1)--(C);
				\node[yshift=3mm] at (BB5) {$p_1$};
				\node[yshift=3mm] at (BB4) {$p_\ell$};
				\node[yshift=1mm,xshift=3mm] at (C) {$\eta_n$};
\end{tikzpicture}
				\label{eq:8.11}
				\end{aligned}
\end{equation}
				\label{example:8.8}
\end{example}
with $\eta_n$ denoting the maximal ideal of $\Z\left[ \frac{1}{n} \right]\cap
\Z_{\R}$.  \\ 
The specialization  picuture of the space
$\lim\limits_{\underset{J}{\longleftarrow}}\xspace_{n}$ is 
\begin{equation}
				\begin{aligned}
\begin{tikzpicture}[baseline=0mm]
				\node at (33mm,15.5mm) {$\cdots$};
				\node at (45mm,15.5mm) {$\cdots$};
				\node at (10mm,17mm) {$\eta$};
				\node at (17.5mm,17mm) {$2$};
				\node at (25mm,17mm) {$3$};
				\node at (30mm,2mm) {$(0)$};
				\node at (40mm,17mm) {$p$};
				\draw (30mm,5mm)--(10mm,15mm);
				\draw (30mm,5mm)--(17.5mm,15mm);
				\draw (30mm,5mm)--(25mm,15mm);
				\draw (30mm,5mm)--(40mm,15mm);
\end{tikzpicture}
				\end{aligned}
				\label{eq:8.12}
\end{equation}
When $n|m$, the map $\xspace_{m}\rightarrow \xspace_n$ is the identity on points, and
for $\mathscr{U}\subseteq \xspace_n$ open:
$\bigo_{\xspace_{n}}(\mathscr{U})\equiv\bigo_{\xspace_{m}}(\mathscr{U})$; but there
are more open neighborhood to $\eta_{m}\in \xspace_{m}$ than there are for $\eta_n\in
\xspace_{n}$. 
\section{Valuations} \label{sec:9}
\begin{definition}
				We say $K\in\cbio_0$ is a \myemph{field} if 
				\begin{equation*}
								K(1)\setminus \left\{ 0 \right\} \equiv K^{*}:=\left\{ a\in K(1), \exists
								a^{-1}\in K(1), a\circ a^{-1}=1\right\}.
				\end{equation*}
				A $\text{sub-Bio}_{0}$ $B\subseteq K$ is than a \myemph{domain}: $\left\{ 0
				\right\}$ is prime and the localization gives embedding 
				\begin{equation}
								B\subseteq B_{(0)} \subseteq K
								\label{eq:9.2}
				\end{equation}
				We define
				\begin{equation}
								\begin{array}[H]{ll}
												B^{+}(n)^{\perp} = \left\{p\in K^{-}(n), p\circ B^{+}(n)\subseteq
												B(1)\right\} \\\\
												B^{-}(n)^{\perp} = \left\{q\in K^{+}(n), B^{-}(n)\circ q\subseteq
												B(1)\right\} 
								\end{array}
								\label{eq:9.3}
				\end{equation}
        For fields $k\subseteq K \in \cbio_{0}^{t} $ (with involution!) we define the valuation-sub-bios
				\begin{equation}
								\text{Val}(K/k):=
								\left\{
								\begin{array}[H]{l}
												B\in\cbio_{0}^{t}, \; k\subseteq B\subseteq B_{(0)}\equiv K, \\\\
												B^{\pm}(n)^{\perp}=B^{\mp}(n) \\\\
												\forall x \in K^{\ast}, \; x\in B(1)\;\; \text{or}\;\; x^{-1}\in
												B(1)
								\end{array}\right\}
								\label{eq:9.4}
				\end{equation}
We define the real-valued valuations
\begin{equation}
								\text{Val}_{\R}\left( K/k \right):= \left\{
												\begin{array}[H]{l}
																|\;| = |\;|_{n,\pm}: K^{\pm}(n)\longrightarrow [0,\infty)
																\\\\
																|x_1\circ x_2|_1 = |x_1|_1\cdot |x_2|_1, \quad  |x|_1=0
																\Leftrightarrow x=0. \\\\
																\big|p\circ(p_i)\big|\le |p|\cdot\max |p_i|;\quad |(q_i)\circ
																q|\le |q|\cdot \max|q_i|; \\\\
																|p\circ (q_i)|\le |p|\cdot \max|q_i|;\quad |(p_i)\circ q|\le
																|q|\cdot \max |p_i|; \\\\
																|p|=|p^{t}|;\quad \forall\lambda\in k, |\lambda|\le 1; \\\\
																|p|_{n,-}= \sup\left\{ |p\circ q|_{1}, |q|_{n+}\le 1 \right\} =
																\inf \left\{ |d^{-1}|_{1}, |d\circ p|_{n-}\le 1 \right\} \\\\
																|q|_{n,+}= \sup\left\{ |p\circ q|_{1}, |p|_{n-}\le 1 \right\} =
																\inf \left\{ |d^{-1}|_{1}, |q\circ d|_{n+}\le 1 \right\} 
												\end{array}
								\right\}_{\text{\Huge $/\approx$}}
								\label{eq:9.5}
				\end{equation}
				where $\approx$ is the equivalence relation given by 
				\begin{equation*}
								||\approx||^\sigma \quad , \quad \sigma>0.
				\end{equation*}
				\label{def:9.1}
\end{definition}
We have a map  
\begin{equation}
				\begin{array}[H]{rl}
								\text{Val}_{\R}\left( K/k \right)&\hookedrightarrow{1} \text{Val}\left( K/k \right) \\\\
				|\;| &\xmapsto{\qquad} B_{||}\equiv \left\{ p\in K, |p|\le 1 \right\}
				\end{array}
				\label{eq:9.6}
\end{equation}
and this map is a bijection if for all $B\in\text{Val}\left( K/k \right)$ the set
$K^{\ast}/B^{\ast}$, which is partially-ordered by 
\begin{equation}
				x\le y \Longleftrightarrow x\circ y^{-1}\in B(1)
				\label{eq:9.7}
\end{equation}
is class one and there is an order preserving embedding $|\;|_1:K^{\ast}/B^{\ast}\hookrightarrow
(0,\infty)$. 
\begin{theorem}[Ostrowski Theorem]
				We have 
				\begin{equation}
								\text{Val}\left( \Q/\F \right) = \text{Val}_{\R}\left( \Q/\F \right)=
								\left\{ \Q;\Z_{(p)}=\Q\cap \Z_p,p\;\text{prime}; \Q\cap \Z_{\R} \right\}
								\label{eq:9.8}
				\end{equation}
				For a number field $K$ we have 
				\begin{equation}
								\begin{array}[H]{ll}
												\text{Val}\left( K/\F \right)&=\text{Val}_{\R}\left( K/\F \right) \\\\
								&= 
								
								\left\{
												\begin{array}[H]{l}
												K;\bigo_{K,\pid}=K\cap\hat{\bigo}_{K,\pid}, \pid \subseteq
												\bigo_{K}\;\text{prime}; \vspace{.1cm}\\
				K\cap\sigma^{-1}(\Z_{\C}),\sigma:K\hookrightarrow\C,\sigma\sim\overline{\sigma} 
												\end{array}\right\}
								\end{array}
				\label{eq:9.9}
				\end{equation}
				\label{thm:ostrowski}
\end{theorem}
\begin{proof}
  For a proof see \cite{H17}
\end{proof}
\begin{remark}
				The requirement that $B$ is preserved by the involution of $K$, $|p^{t}|=|p|$, is
				crutial: without it we have also the $\ell_p$-valuations $\Q\cap\left(
								\biop_{\ell_{p}},\biop_{\ell_{q}}
				\right)$, cf. \S\ref{sec:3}
				\label{remark:9.10}
\end{remark}
\section{Props} \label{sec:10}
\begin{definition}
				A \myemph{prop} $\biop$ is a strict-symmetric-monoidal category which is generated
				by one object $x$. 
				\label{def:10.1}
\end{definition}
\noindent Thus we have sets $\biop_{n,m}=\biop\left( x^{\oplus m},x^{\oplus n} \right)$ with an
action of $S_n\times S_m^{\op}$ and composition
\begin{equation}
				\begin{array}[H]{rl}
				\biop_{n,m}\times \biop_{m,\ell}&\xrightarrow{\qquad}  \biop_{n,\ell} \\\\
				f,g &\xmapsto{\qquad} f\circ g
				\end{array}
				\label{eq:10.2}
\end{equation}
which is $S_{m}$-invariant, $S_{n}\times S_{\ell}^{\op}$-covariant, and associative
\begin{equation}
				(f\circ g )\circ h = f\circ (g\circ h)
				\label{eq:10.3}
\end{equation}
and unital 
\begin{equation}
				1_{n}\circ f = f = f\circ 1_{m} \; , \; f\in \biop_{n,m}\; , \;
				1_{n}\in\biop_{n,n}
				\label{eq:10.4}
\end{equation}
We also assume $\biop_{0,m}=\left\{ 0 \right\}$, $\biop_{n,0}=\left\{ 0 \right\}$, and
$0$ is an initial and final object of $\biop$. The monoidal structure is given by 
\begin{equation}
				\begin{array}[H]{rl}
								\biop_{n_1,m_1}\times \biop_{n_2,m_2} &\xrightarrow{\qquad}
								\biop_{n_1+n_2,m_1+m_2} \\\\
								f_1,f_2\xmapsto{\qquad}f_1\oplus f_2
				\end{array}
				\label{eq:10.5}
\end{equation}
It is $S_{n_1}\times S_{n_2}\subseteq S_{n_1+n_2}$ and $S_{m_1}\times S_{m_2}\subseteq
S_{m_1+m_2}$ covariant, functorial 
\begin{equation}
				\begin{array}[H]{c}
				(f_{1}\oplus f_2)\circ (g_1\oplus g_2) = (f_1\circ g_1)\oplus (f_2\circ g_2) \\\\
				1_{n+m}= 1_n\oplus 1_m
				\end{array}
				\label{eq:10.6}
\end{equation}
strict monoidal 
\begin{equation}
				\begin{array}[H]{c}
				(f_1\oplus f_2)\oplus f_3 = f_1\oplus (f_2\oplus f_3) \\
				f\oplus \text{id}_0=f = \text{id}_0\oplus f
				\end{array}
				\label{eq:10.7}
\end{equation}
and symmetric 
\begin{equation}
				f_2\oplus f_1 = \tau_{n_1,n_2}\circ (f_1\oplus f_2)\circ \tau_{m_2,m_1}
				\label{eq:10.8}
\end{equation}
with $\tau_{n,m}\in S_{n+m}$ given by
\begin{equation}
				\tau(i)=\left\{ 
								\begin{array}[H]{lc}
												n+i & i\le m \\\\
												i-m & i>m
								\end{array}
				\right.
				\label{eq:10.9}
\end{equation}
pictorially, 
\begin{equation}
				\begin{aligned}
\begin{tikzpicture}[baseline=0mm]
				\node at (5mm,5mm) {$n+m$};
				\node at (5mm,20mm) {$n+1$};
				\node at (5mm,30mm) {$n$};
				\node at (5mm,45mm) {$1$};
				\node at (5mm,38mm) {$\vdots$};
				\node at (5mm,14mm) {$\vdots$};

				\node at (50mm,5mm) {$m+n$};
				\node at (50mm,20mm) {$m+1$};
				\node at (50mm,30mm) {$m$};
				\node at (50mm,45mm) {$1$};
				\node at (50mm,38mm) {$\vdots$};
				\node at (50mm,14mm) {$\vdots$};

				\draw (0mm,43mm).. controls (35mm,45mm) and (35mm,12mm) .. (55mm,18mm);
				\draw (0mm,28mm).. controls (35mm,30mm) and (35mm,-3mm) .. (55mm,3mm);
				\draw (0mm,3mm).. controls (20mm,-3mm) and (35mm,33mm) .. (55mm,28mm);
				\draw (0mm,18mm).. controls (20mm,12mm) and (35mm,48mm) .. (55mm,42.5mm);
\end{tikzpicture}
				\end{aligned}
				\label{eq:10.10}
\end{equation}
For props $\biop$ and $\mathscr{Q}$ we put 
\begin{equation}
        \prop(\biop,\mathscr{Q})= 
	\left\{ 
		\begin{array}[H]{l}
                	f_{n,m}\in \text{Set}(\biop_{n,m},\mathscr{Q}_{n,m}),  f(\sigma\cdot p
			\cdot \sigma^{\prime})= \sigma\cdot f(p)\cdot \sigma, f(0)=0, \\\\
			f(p_{1}\circ p_{2}) = f(p_{1})\circ f(p_{2}), f(1)=1, f(p_1\oplus
			p_2) = f(p_1)\oplus f(p_2)
		\end{array}
	\right\}
	\label{eq:10.11}
\end{equation}
so that we have a category $\prop$. 
\begin{remark}
(cf. \ref{example:1.14}-\ref{remark:1.16}). Given a strict symmetric monoidal category
$\E$ and object $x\in\E^{0}$ we have the prop $\text{End}_{\E}(x)$ with 
\begin{equation}
				\text{End}_{\E}(x)_{n,m}:=\E(x^{\oplus m}, x^{\oplus n}). 
				\label{eq:10.13}
\end{equation}
Replacing everywhere Set by $\E$ we obtain the category of props in
$\E:\prop(\E)$. \\
The category $\prop$ has an involution $\biop\mapsto \biop^{\text{op}}$ with
\begin{equation}
				\biop^{\text{op}}_{n,m}:=\biop_{m,n}
  \label{eq:14}
\end{equation}
We have therefore the props with an ivnolution $p\mapsto p^{t}$,
$\biop_{n,m}\xrightarrow{\sim}\biop_{m,n}$ and the category $\prop^{t}$ (with maps
preserving the involutions). 
\label{remark:10.12}
\end{remark}
\begin{definition}
				A prop $\biop$ will be called \myemph{central} if the monoid $\biop_{1,1}$ is
				commutative and \myemph{central}:
				\begin{equation}
								a\cdot p := ( \overset{n}{\oplus}\, a)\circ p \equiv p\circ
								( \overset{m}{\oplus}\,a )\quad \text{for $a\in\biop_{1,1}$,
								$p\in\biop_{n,m}$}.
								\label{eq:10.16}
				\end{equation}
				A prop $\biop$ will be called \myemph{commutative} if it is central and we have
				for $b\in\biop_{1,k}$, $p\in\biop_{n,m} $, $d\in\biop_{k,1}$ 
				\begin{equation}
								(b\circ d)\cdot p \equiv ( \overset{n}{\oplus}\,b )\circ
								\sigma_{n,k}\cdot ( \overset{k}{\oplus}\,p )\circ
								\sigma_{k,m}\circ ( \overset{m}{\oplus}\, d )
								\label{eq:10.17}
				\end{equation}
				A prop $\biop$ will be called \myemph{totally-commuative} if we have for
				$p\in\biop_{p_0,p_1}$, \break $q\in\biop_{q_0,q_1}$, 
				\begin{equation}
								p\otimes q := (\overset{q_0}{\oplus}\, p)\circ \sigma_{q_0,p_1}\circ
                (\overset{p_1}{\oplus}\,q) \equiv \sigma_{q_0, p_0}(\overset{p_0}{\oplus}\,
								q)\sigma_{p_0,q_1}(\overset{q_1}{\oplus}\, p)\sigma_{q_1,p_1}
								\label{eq:10.18}
				\end{equation}
				(here $\sigma_{n,m}\in S_{n\cdot m}$ is the permutation (\ref{eq:4.3})). 
				\label{definition:10.15}
\end{definition}
We let 
\begin{equation}
				\begin{aligned}
\begin{tikzpicture}[baseline=0mm]
				\node at (5mm,5mm) {$\ctprop   $};
				\node at (15mm,5mm) {$\subseteq$};
				\node at (24mm,5mm) {$\cprop$};
				\node at (32mm,5mm) {$\subseteq  $};
				\node at (40mm,5mm) {$\prop  $};
        \node at (38mm,11mm) {$\mathscr{P}$};
				\draw [{Bar[width=3pt]}-{Stealth[sep,length=0.8mm,width=1.2mm]}] (36mm,11mm) -- (31mm,11mm);
				\node at (29mm,11.5mm) {$\biop^C$};
				\draw[arrows={->[slant=-.6]>[slant=-.4]}] (36mm,6.9mm).. controls (32mm,8mm) .. (28mm,6.9mm);
				\draw[arrows={->[slant=-.6]>[slant=-.4]}] (19mm,6.9mm).. controls (15mm,8mm) .. (11mm,6.9mm);
				\node at (18mm,11mm) {$\biop$};
				\draw [{Bar[width=3pt]}-{Stealth[sep,length=0.8mm,width=1.2mm]}] (15.5mm,11mm) --
				(11mm,11mm);
				\node at (9mm,11.5mm) {$\biop^T$};
\end{tikzpicture}
\label{eq:10.19}
				\end{aligned}
\end{equation}
denote the full subcategory of (resp. totally-) commutative props, and we let
$\biop\xtworightarrow{\quad}\biop^{C}$ (resp. $\biop\xtworightarrow{\quad}\biop^{T}$)
denote the maximal (totally-) commutative quotient of $\biop$ giving the left adjoints of
the inclusions (\ref{eq:10.19}). 
\begin{equation}
\begin{aligned}
\begin{tikzpicture}[baseline=0mm]
				\node at (15mm,20mm) {$\ctprop$};
				\node at (41mm,30mm) {$\cprop$};
				\draw[<<-] (22mm,23mm)--(35mm,30.5mm); 
				\node at (65mm,45mm) {$\prop$};
				\draw[{Hooks[right,length=5,width=6]}->] (23mm,20mm)--(35mm,27.5mm); 
				\draw[{Hooks[right,length=5,width=6]}->]  (46mm,32mm)--(60mm,42.5mm);
				\draw[<<-] (44mm,35mm)--(60mm,45.5mm); 
				\node at (15mm,0mm) {$\ctbio_0$};
				\draw[{Hooks[right,length=5,width=6]}->] (23mm,20mm)--(35mm,27.5mm); 
				\node at (41mm,10mm) {$\cbio_0$};
				\draw[<<-] (22mm,3mm)--(35mm,10.5mm); 
				\node at (65mm,25mm) {$\bio_0$};
				\draw[{Hooks[right,length=5,width=6]}->] (23mm,0mm)--(35mm,7.5mm); 
				\draw[{Hooks[right,length=5,width=6]}->]  (46mm,11mm)--(60mm,21.5mm);
				\draw[<<-] (44mm,15mm)--(60mm,25.5mm); 
				\draw[->] (17mm,18mm)--(17mm,3mm); 
				\draw[<-] (14mm,18mm)--(14mm,3mm); 
				\node at (11mm,10mm) {$\mathcal{F}_{T}$};
				\node at (19mm,10mm) {$U$};
				\draw[->] (42mm,28mm)--(42mm,13mm); 
				\draw[<-] (39mm,28mm)--(39mm,13mm); 
				\node at (36mm,21mm) {$\mathcal{F}_{C}$};
				\node at (44mm,21mm) {$U$};
				\draw[->] (66mm,42mm)--(66mm,28mm); 
				\draw[<-] (63mm,42mm)--(63mm,28mm); 
				\node at (60mm,34mm) {$\mathcal{F}$};
				\node at (68mm,35mm) {$U$};
\end{tikzpicture}
\end{aligned}
\label{eq:10.20}
\end{equation}
with the frogetfull functor $U$ 
\begin{equation}
				\begin{array}[H]{c}
								\left( U\biop \right)^{-}(n) = \biop_{1,n} \quad , \quad
								(U\biop)^{+}(n)=\biop_{n,1} \\\\
								p\circ (p_{i}):= p\circ {\oplus} p_i, (q_i)\circ q:=
								({\oplus} q_{i})\circ q \quad \text{etc.}
				\end{array}
				\label{eq:10.21}
\end{equation}
The left adjoints $\mathcal{F}\biop$ give the free prop generated by the bio $\biop$. Note
that when $\biop$ is (totally) commutative, the elements of $\left( \mathcal{F}_{C}\biop
\right)_{n,m}$, (resp. $\mathcal{F}_{T}\biop$), can be written, non-uniquely, as 
\begin{equation}
				\begin{array}[H]{c}
				\text{``}\;\; \left( \oplussum\limits_{i=1}^{n} p_i \right)\circ \sigma\circ \left(
				\oplussum\limits_{j=1}^{m}q_{j} \right) \;\; \text{''} \\\\
				p_{i}\in \biop^{-}(k_i), \; q_{j}\in \biop^{+}(\ell_j), \; \sigma\in S_{L}, \;
				L=k_1+\cdots + k_n = \ell_1+\cdots + \ell_m.
				\end{array}
				\label{eq:10.22}
\end{equation}
Moreover, we have
\begin{equation}
				\Big( \overset{n}{\oplus} p_i \Big)\circ \sigma \circ \Big(
				\overset{m}{\oplus}q_{j} \Big) \simeq
        \Big(\overset{n^{\prime}}{\hspace{-1mm}\oplus} p^{\prime}_i\Big)\circ
				\sigma^{\prime}\circ \Big(\overset{m^{\prime}}{\hspace{-1mm}\oplus}
        q_{j}^{\prime}\Big)
				\label{eq:10.23}
\end{equation}
iff we have $\big( \overset{n}{\oplus}\varphi(p_i)\big)\circ \sigma \circ
\big(\overset{m}{\oplus}\varphi(q_j)\big) =  \big(
\overset{n^{\prime}}{\hspace{-1mm}\oplus}\varphi(p^{\prime}_i)\big)\circ \sigma^\prime \circ
\big(\overset{m^{\prime}}{\hspace{-1mm}\oplus}\varphi(q_j^\prime)\big)$
for all $\varphi\in\bio(\biop,U\mathscr{Q})$ and all $\mathscr{Q}\in \cprop$ (resp. $\ctprop$).\\\\
For a rig 
$R$ we have the prop $\biop_{R}$ with $(\biop_{R})_{n,m}$ the $n$ by $m$ matrices with
entries in $R$, composition is matrix multiplication, and $\oplus$ is block-direct sum of
matrices. When $R$ has an involution $\biop_R$ has an involution; for $A=(a_{ij})$ in
$(\biop_{R})_{n,m}$ we have $A^{t}=(b_{ji})$, $b_{ji}=a_{ij}^{t}$. When $R$ is
commutative $\biop_{R}$ is totally-commutative and the diagram 
\begin{equation}
				\begin{aligned}
				\begin{tikzpicture}[baseline=0mm]
				\node at (7mm,18mm) {$\crig$};
				\draw[{Hooks[right,length=5,width=6]}->] ($(13mm,23mm)-(0,3mm)$)--($(27mm,32mm)-(0,3mm)$); 
				\draw[{Hooks[right,length=5,width=6]}->] ($(13mm,20mm)-(0,3mm)$)--($(27mm,12mm)-(0,3mm)$); 
				\node at (35mm,30mm) {$\ctprop$};
				\node at (34mm,8mm) {$\ctbio_0$};
				\draw[<-] (36.5mm,11mm) -- (36.5mm,28mm);
				\draw[->] (33mm,11mm) -- (33mm,28mm);
				\node at (30mm,19mm) {$\mathcal{F}_{T}$};
				\node at (38.5mm,19mm) {$U$};
\end{tikzpicture}
				\end{aligned}
				\label{eq:10.24}
\end{equation}
is commutative. \\
For a $\prop$ $\biop$ we have the groups of invertible elements
\begin{equation}
				\gl_{n}(\biop):= \left\{ a\in \biop_{n,n}, \exists a^{-1}\in\biop_{n,n},
				a^{-1}\circ a= a\circ a^{-1}=1_{n} \right\}
				\label{eq:10.25}
\end{equation}
and homomorphisms
\begin{equation}
				\begin{array}[H]{cl}
								\gl_{n}(\biop)\times \gl_{m}(\biop) &\xrightarrow{\qquad}
								\gl_{n+m}(\biop) \\\\
								a\quad , \quad b  &\xmapsto{\qquad} a\oplus b
				\end{array}
				\label{eq:10.26}
\end{equation}
\begin{equation}
  \label{eq:10.27}
\begin{array}[h]{l}
  \tag*{(10.27)\; E.g. } \stepcounter{equation} \\
  \gl_n(\F)\equiv S_n\subseteq \gl_n(\F[\pm 1])\equiv (\pm 1)^n\rtimes S_n \subseteq \gl_n(\Z_{\R})\equiv O(n)\subseteq \gl(\Z_{\C})\equiv U(n)
\end{array}
\end{equation}
\section{$\prop$-Schemes} \label{sec:11}
For a commutative prop $\biop\in\cprop$ we have the compact sober topological space
$\text{spec}(\biop)\equiv \text{spec}(U\biop)$. Since $\biop_{1,1}$ acts centrally on the sets
$\biop_{n,m}$ there is no problem with localizations with respect to multiplicative sets
$S\subseteq\biop_{1,1}$ and we get cannonical map 
\begin{equation}
				\phi_{S}:\biop\xrightarrow{\qquad} S^{-1}\biop
				\label{eq:11.1}
\end{equation}
In particular we have the localizations
\begin{equation}
				\begin{array}[H]{l}
				\biop_{\pid}:=S_{\pid}^{-1}\biop \quad , \quad S_{\pid}=\biop_{1,1}\setminus \pid
								\quad, \quad \pid\subseteq \biop_{1,1}\;\; \text{prime} \\\\
								\biop_{f}:= S_{f}^{-1}\biop\quad , \quad S_{f}=\left\{ 1,f,f^{2},\cdots
								f^{n}\cdots \right\}\quad , \quad f\in\biop_{1,1}.
				\end{array}
				\label{eq:11.2}
\end{equation}
Moreover, commutativity imply that there is a sheaf of commutative props $\bigo_{\biop}$
over $\text{spec}(\biop)$ with  stalks 
\begin{equation}
				\bigo_{\biop, \pid} =
				\lim\limits_{\mathscr{U}\ni\pid}\bigo_{\biop}(\mathscr{U})\xrightarrow{\sim}
				\biop_{\pid}
				\label{eq:11.3}
\end{equation}
and with global sections
\begin{equation}
				\biop_{f}\xrightarrow{\quad\sim\quad} \bigo_{\biop}\big(D(f)\big).
				\label{eq:11.4}
\end{equation}
We get a functor 
\begin{equation}
				\text{spec}: \big(\cprop)^{\op}\xrightarrow{\qquad}\cprop_{\text{loc}}/\top
				\label{eq:11.5}
\end{equation}
with $\cprop_{\text{loc}}/\top$ the category with objects $(\xspace,\bigo_{\xspace})$,
$\xspace\in\top$, $\bigo_{\xspace}$ a sheaf of commutative props with local stalks, and maps of
prop-spaces inducing local maps on stalks. We get the category of \myemph{Prop-Schemes}
$\psch$: 
\begin{equation}
				(\cprop)^{\op}\subseteq \psch \subseteq \cprop_{\text{loc}}/\top , 
				\label{eq:11.6}
\end{equation}
it is the full subcategory of $\cprop_{\text{loc}}/\top$ consisting of $(\xspace,
\bigo_{\xspace})$ such that for some open cover $\xspace=\bigcup\limits_{i}
\mathscr{U}_{i}$ 
\begin{equation}
				\big( \mathscr{U}_{i}, \bigo_{\xspace}\big|_{\mathscr{U}_{i}}\big) \cong
				\text{spec}\big(\bigo_{\xspace}(\mathscr{U}_{i})\big)
				\label{eq:11.7}
\end{equation}
The adjunction (\ref{eq:10.20}) prolongs to an adjunction, over $\top$, under Sch, 
\begin{equation}
				\begin{aligned}
				\begin{tikzpicture}[baseline=0mm]
				\node at (9mm,18.5mm) {$\sch$};
				\draw[{Hooks[right,length=5,width=6]}->] ($(13mm,23mm)-(0,3mm)$)--($(27mm,32mm)-(0,3mm)$); 
				\draw[{Hooks[right,length=5,width=6]}->] ($(13mm,20mm)-(0,3mm)$)--($(27mm,12mm)-(0,3mm)$); 
				\node at (35mm,30mm) {$\psch$};
				\node at (34mm,8mm) {$\gsch$};
				\draw[<-] (36.5mm,11mm) -- (36.5mm,28mm);
				\draw[->] (33mm,11mm) -- (33mm,28mm);
				\node at (30mm,19mm) {$\mathcal{F}_{C}$};
				\node at (38.5mm,19mm) {$U$};
				\node at (60mm,18.5mm) {$\top$};
				\draw[->] (41mm,9mm) -- (55mm,17.5mm);
				\draw[->] (41mm,28mm) -- (55mm,19.5mm);
\end{tikzpicture}
				\end{aligned}
				\label{eq:11.8}
\end{equation}
The rest of section \S \ref{sec:7} work verbatim for prop-schemes. Similarly, the
discussion of limits in \S \ref{sec:8} works exactly the same with prop-schemes and we
have the category of pro-$\prop$ schemes
\begin{equation}
				\begin{aligned}
\begin{tikzpicture}[baseline=0mm]
				\node at (40mm,0mm) {$\text{pro-}\biop\text{Sch}$};
				\node at (25mm,25mm) {$\biop\text{Sch}$};
				\node at (65mm,25mm) {$\cprop_{\text{loc}}/\top$};
				\draw[{Hooks[right,length=5,width=6]}->] (32mm,25mm)--(52mm,25mm);
				\draw[{Hooks[left,length=5,width=6]}->] (25mm,22mm)--(35mm,3mm);
				\draw[->] (45mm,3mm) -- (55mm,22mm);
				\node at (55mm,10mm) {$\lim\limits_{\longleftarrow}$};
\end{tikzpicture}
				\end{aligned}
				\label{eq:11.9}
\end{equation}
\begin{example}
				Using the notations of (\ref{example:8.8}), we have the ``\myemph{compactified
				$\text{spec}(\Z)$}'' with the same underlying topological spaces as in
				(\ref{eq:8.10}), (\ref{eq:8.11}), but with the sheaf of props $\bigo_{\Z}$ with
				$(\bigo_{\Z})_{n,m}(\mathscr{U})$ the $n$ by $m$ $\Q$-valued matrices \break $a=(a_{ij})$,
				$a_{ij}\in \Q$ such that 
				\begin{equation}
								a\circ \Z_{p}^{m}\subseteq \Z_{p}^{n} \quad , \quad \text{all $p\in \mathscr{U}$},
								\label{eq:11.11}
				\end{equation}
				and moreover, 
				\begin{equation}
								a\circ\Z_{\R}^{m}\subseteq \Z_{\R}^{n} \quad \text{if the real prime
								$\eta\in \mathscr{U}$}
								\label{eq:11.12}
				\end{equation}
				(with $\Z_{\R}^{n}=\left\{ (x_{i})\in\R^{n}, \sum |x_i|^2\le 1 \right\}$ the unit
				$\ell_2$-balls). The global sections
				$\bigo_{\Z}(\overline{\text{spec}\Z})_{n,m}$ are the $n$ by $m$ $\Z$ valued
				matrices satisfying (\ref{eq:11.12}), which are the matrices having in each column
				and in each raw at most one non-zero term with entry $1$ or $-1$:
				\begin{equation}
								\bigo_{\Z}(\overline{\text{spec}\Z})\equiv \F[\pm 1]
								\label{eq:11.13}
				\end{equation}
\end{example}
\section{The Witt ring and Frobenius correspondences}
For a prop $\biop\in \cprop$ we have the group $\gl_{n}(\biop)$ of invertible elements of
the monoid $\biop_{n,n}$ and it acts on $\biop_{n,n}$ by conjugation; we let $[p]$ denote
the conjugacy class of $p\in\biop_{n,n}$ and we let $[\biop_{n,n}]$ denote the collection
of these conjugacy classes. We have an embedding
\begin{equation}
				\left[ \biop_{n,n} \right]\xrightarrow{\qquad} \left[ \biop_{n+1,n+1} \right]\;\;
				, [p]\xmapsto{\qquad}[p\oplus 0], 
				\label{eq:12.1}
\end{equation}
and we take the direct limit:
\begin{equation}
				[\biop]:=\lim\limits_{\overset{\rightarrow}{n}}\left[ \biop_{n,n} \right]
				\label{eq:12.2}
\end{equation}
We have a commutative monoid structure on $[\biop]$ via 
\begin{equation}
				[p_1]+[p_2] := \left[ p_1\oplus p_2 \right]
				\label{eq:12.3}
\end{equation}
and applying $K$ (Grothendieck localization of addition) we get an abelian group, the Witt group:
\begin{equation}
				\mathcal{W}(\biop):=K \left( [\biop] \right). 
				\label{eq:12.4}
\end{equation}
This group has the Frobenius endomorphisms given by 
\begin{equation}
				F_n:\mathcal{W}(\biop)\rightarrow \mathcal{W}(\biop)\;\;, \;\; F_n[p]=[p^n]=
				\left[ \underbrace{p\circ p\circ \cdots\circ p}_{n} \right]
				\label{eq:12.5}
\end{equation}
We thus have a functor 
\begin{equation}
				\mathcal{W}: \prop\xrightarrow{\qquad} (Ab)^{\N}
				\label{eq:12.6}
\end{equation}
When the prop $\biop$ is \myemph{totally} commutative, we have multiplication on $[\biop]$ via, cf.
(\ref{eq:10.18}), 
\begin{equation}
				[p_1]\cdot [p_2]= [p_1\otimes p_2]
				\label{eq:12.7}
\end{equation}
making $[\biop]$ into a commutative rig, and so $\mathcal{W}(\biop)$ is a commutative ring
and the $F_n$'s are ring endomorphisms: 
\begin{equation}
				\mathcal{W}:\ctprop\xrightarrow{\qquad}(\cring)^{\N}
				\label{eq:12.8}
\end{equation}
We can summarize our constructions in the following diagram:
\begin{equation}
  \begin{aligned}
  \begin{tikzpicture}[baseline=0mm]
    \node at (50mm,5mm) {$\cring$};
    \node at (90mm,5mm) {$\left( \cring_{\text{loc}}\Big/\top \right)^{\text{op}}$};
    \draw[{Hooks[right,length=5,width=6]}->] (56mm,3.5mm)--(76mm,3.5mm); 
    \draw[<<-] (56mm,6mm)--(76mm,6mm); 
    \node at (65mm,1mm) {$\text{spec}$};
    \node at (65mm,8mm) {$\Gamma$};
    \draw[{Hooks[left,length=5,width=6]}->] (48mm,8mm)--(44mm,18mm);
    \node at (44mm,20mm) {$\ctbio_0$};
    \node at (37mm,35mm) {$\ctprop_0$};
    \node at (8mm,36mm) {$\left( \cring \right)^{\N^{+}}$};
    \draw[<-] (15mm,35mm)--(28mm,35mm);
    \node at (23mm,37mm) {$\mathscr{W}$};
    \draw[{Hooks[right,length=5,width=6]}->] (40mm,24mm)--(36mm,33mm);
    \draw[<<-] (42mm,24mm)--(38mm,33mm);
    \node at (35mm,28mm) {$\mathcal{F}_{T}$};
    \node at (42mm,30mm) {$U$};
    \draw[{Hooks[right,length=5,width=6]}->] (49mm,24mm)--(56mm,33mm);
    \draw[<<-] (45mm,24mm)--(53mm,33mm);
    \node at (55mm,35mm) {$\cbio_0$};
    \node at (45mm,55mm) {$\cprop_0$};
    \node at (90mm,35mm) {$\left(\cbio_{\text{loc}}/\top\right)^{\text{op}}$};
    \draw[{Hooks[right,length=5,width=6]}->] (88mm,9mm)--(88mm,31mm); 
    \node at (91mm,55mm) {$\left( \cprop_{\text{loc}} /\top\right)^{\text{op}}$};
    \draw[{Hooks[right,length=5,width=6]}->] (53mm,53.5mm)--(75mm,53.5mm); 
    \draw[<<-] (53mm,56mm)--(75mm,56mm); 
    \node at (65mm,51mm) {$\text{spec}$};
    \node at (65mm,58mm) {$\Gamma$};
    \draw[{Hooks[right,length=5,width=6]}->] (61mm,33.5mm)--(76mm,33.5mm); 
    \draw[<<-] (61mm,36mm)--(76mm,36mm); 
    \node at (69mm,31mm) {$\text{spec}$};
    \node at (69mm,38mm) {$\Gamma$};
    \draw[{Hooks[right,length=5,width=6]}->] (86mm,38mm)--(86mm,51mm); 
    \draw[<<-] (88mm,38mm)--(88mm,51mm); 
    \node at (91mm,45mm) {$U$};
    \node at (83mm,45mm) {$\mathcal{F}_{C}$};
    \draw[<<-] (55mm,38mm)--(50mm,52mm); 
    \draw[{Hooks[right,length=5,width=6]}->] (53mm,38mm)--(48mm,52mm); 
    \node at (55mm,47mm) {$U$};
    \node at (48mm,44mm) {$\mathcal{F}_{C}$};
    \draw[<<-] (35mm,38mm)--(40mm,52mm); 
    \draw[{Hooks[right,length=5,width=6]}->] (38mm,38mm)--(43mm,52mm); 
  \end{tikzpicture}
  \end{aligned}
  \label{eq:12.9a}
\end{equation}
Here $\left( \cring \right)^{\N^{+}}$ is the category of commutative rings with an action
of the multiplicative monoid $\N^{+}$ of positive natural numbers (the free commutative
monoid on the set of primes), and $\cring_{\text{loc}}/\top$ is the usual category of
locally ringed spaces. \\
Given a prop-scheme $(\xspace,\bigo_{\xspace})$ we can apply $\mathcal{W}$ to the props
$\bigo_{\xspace}(\mathscr{U})$ to obtain a sheaf of abelian groups
$\mathcal{W}(\bigo_{\xspace})$ over
$\xspace$. When all the $\bigo_{\xspace}(\mathscr{U})$ are totally commutative
$\mathcal{W}(\bigo_{\xspace})$ is a sheaf of commutative rings with Frobenius
endomorphisms. \\\\
Given an ordinary commutative ring $R\in\cring$, we have the
				associated totally-commutative prop $\biop_{R}\in\ctprop$ and
        $\mathcal{W}(\biop_{R})$ is the ring of rational Witt vectors, cf. \cite{A},
\begin{equation}
	\begin{array}[H]{rc}
		\displaystyle \mathcal{W}(\biop_{R}) & \hookedrightarrow{1}1+x\cdot
                R\left[ [x] \right] \\\\
                \displaystyle \left[ p_1 \right]-\left[ p_2 \right]
                &\xmapsto{\qquad} \frac{\det(1-x\cdot p_1)}{\det (1-x\cdot p_2)}.
	\end{array}
	\label{eq:12.10}
\end{equation}
For a domain $R$, with fraction field $K$, let $\overline{K}$ be an algebraic
closure of $K$, \break
$G_{K}=\text{Gal}\left( \overline{K}/K \right)$,
$\overline{R}$ the integral closure of $R$ in $\overline{K}$, than
$\mathcal{W}(\biop_{R})$ is the free abelian group on monic irreducible
polynomials $f(x)\not= x$, via the correspondence 
\begin{equation}
	f(x)=x^{n}+a_1x^{n-1}+\cdots + a_{n}\longleftrightarrow
	\tilde{f}(x)=x^{\deg f}\cdot f\left( \frac{1}{x} \right)=1+a_1x+\cdots
	+a_nx^{n}
	\label{eq:12.11}
\end{equation}
				Alternatively, $\mathcal{W}(\biop_{R})$ is the free abelian group on $G_{K}$
				orbits in $\overline{R}\setminus\left\{ 0 \right\}$: 
				\begin{equation}
								\mathcal{W}(\biop_{R}) = \oplussum\limits_{[\alpha]\in \left(
																\overline{R}\setminus \left\{ 0 \right\}
                            \right) / G_{K} } \Z[\alpha]
								\label{eq:12.12}
				\end{equation}
        with $[\alpha]=\{\alpha=\alpha_{1},\cdots , \alpha_{n}\}$ the $G_{K}$-conjugates of $\alpha$.
				The multiplication in $\mathcal{W}(\biop_{R})$ is given by multiplying orbits:
				\begin{equation}
								[\alpha]\cdot[\beta]=\left\{ \alpha_{1},\cdots , \alpha_{n} \right\}\cdot
								\left\{ \beta_{1},\cdots , \beta_{m} \right\}=\left\{ \alpha_{i}\cdot
								\beta_{j} \right\}
								\label{eq:12.13}
				\end{equation}
\section{The ring $\mathcal{W}=\mathcal{W}(\overline{\text{spec} \Z})$}
For the compactified $\overline{\text{spec}(\Z)}=\text{spec}(\Z)\cup \left\{ \eta
				\right\}$, the condition (\ref{eq:11.12}) at the real prime implies that all the
				eigenvalues of $a$ are in the unit disc, hence the global sections
				$\mathcal{W}(\overline{\text{spec}(\Z)})$ is the free abelian group on
				$G_{\Q}$-orbits of algebraic integers $[\alpha]=\left\{ \alpha_1,\cdots ,
				\alpha_n \right\}$ with $|\alpha_{i}|\le 1$, hence they are roots of unity, and so
	\begin{equation}
		\mathcal{W}(\overline{\text{spec}\Z})\equiv \oplussum\limits_{n\ge 1}\Z\cdot
		\phi_{n} = \oplussum\limits_{n\ge 1}\Z[\bbmu_{n}^{\ast}]
		\label{eq:12.14}
	\end{equation}
				with $\phi_{n}$ the cyclotomic polynomial with roots the primitive $n$-th roots of unity
        $\bbmu_{n}^{*}\cdot$. The multiplication is given by
        \begin{equation}
          \label{eq:12.15} 
                    \phi_n\cdot \phi_m = \phi_{n\cdot m} \quad \text{for}\quad (n,m)=1, 
        \end{equation}
        \begin{equation}
          \begin{array}[H]{ll}
            \phi_{p^{n}}\cdot \phi_{p^{m}} = 
            \left\{\begin{array}[H]{lc}
                (1-p^{-1})p^{m}\cdot \phi_{p^n} & n>m\ge 1 \\\\
                (1-p^{-1})p^{n}\cdot \left[ \phi_{p^{n}}+\phi_{p^{n-1}}+\cdots
                +\phi_{p}+\phi_{1} \right] - p^{n-1}\cdot \phi_{p^{n}} & n=m\ge 1
          \end{array}\right.
          \end{array} 
          \label{eq:12.16}
        \end{equation}
        The algebra $\mathcal{W}=\mathcal{W}(\overline{spec \Z})$ is the union of its finite subalgebras
        \begin{equation}
          \label{eq:12.17}
          \mathcal{W}=\bigcup_{N\geq 1}\mathcal{W}_N \quad , \quad \mathcal{W}_N=\oplussum_{d\mid N}\Z \phi_d
        \end{equation}
        and is the tensor product over all primes (with respect to $\phi_1=1$)
        \begin{equation}
          \label{eq:12.18}
          \mathcal{W} = \otimes_p \mathcal{W}_{p^\infty} \quad , \quad \mathcal{W}_{p^\infty}=\oplus_{\ell\geq 0}\Z \phi_{p^\ell},
        \end{equation}
        with $\mathcal{W}_{p^n}/\mathcal{W}_{p^{n-1}}$ being an integral quadratic extension  (\ref{eq:12.16}). \\
        The Frobenius ring endomorphisms are  completely multiplicative
        \begin{equation}
          \label{eq:2.19}
          F_{m_1}\circ F_{m_2}= F_{m_1\cdot m_2}
        \end{equation}
        and are given for $p$ prime by
        \begin{equation}
          \label{eq:12.20}
          F_p\phi_{n}= \left\{
          \begin{array}[h]{ll}
            \phi_n &  p \not\mid n \\\\
            p\cdot \phi_{n/p} & p^2\mid n \\\\
            (p-1)\cdot \phi _{n/p} & p\mid n \; , \; p^2\not\mid n
          \end{array}\right.
      \end{equation}
      or generally by
      \begin{equation}
        \label{eq:12.21}
        F_m\phi_n= (m,n)\cdot \Big( \prod\limits_{\begin{array}[h]{c}p\mid (m,n) \\ p \not\mid n/(m,n)\end{array}}(1-p^{-1})\Big) \cdot \phi_{n/(m,n)}\;\; , \;\; (m,n)=\gcd(m,n)
      \end{equation}
      We have a cannonical ring homomorphism
      \begin{equation}
        \label{eq:12.22}
          \begin{array}[h]{r}
            t_1=\text{tr}: \mathcal{W}\xrightarrow{\quad}{} \Z \\
           \text{$[a]$} \xmapsto{\quad}{} \text{tr}\; a
          \end{array}
        \end{equation}
        and
        \begin{equation}
          \label{eq:12.23}
          \text{tr}(\phi_n) = \sum_{\xi\in\bbmu_n^{\ast}}\xi=\mu(n)
        \end{equation}
is the M\"obius function. We obtain the homomorphisms
\begin{equation}
  \label{eq:12.24}
  \begin{array}[h]{ll}
    t_m=\text{tr}\circ F_m : & \mathcal{W}\xrightarrow{\quad} \Z \\
    \text{$[a]$}\xmapsto{\quad} \text{tr}(a^m)
  \end{array}
\end{equation}
and
\begin{equation}
  \label{eq:12.25}
  t_m(\phi_n) = \text{tr}(F_m\phi_n)=\sum_{\xi\in\bbmu_n^{\ast}}\xi^m=C_n^m
\end{equation}
with the Ramanujan sums
\begin{equation}
  \label{eq:12.26}
  C_n^m= \mu \left(\frac{n}{(n,m)}\right)\cdot \frac{\varphi(n)}{\varphi(\frac{n}{(n,m)})} \;\; , \; \; n\geq 1 \;\; , \;\; m\in
  \left(\Z/n\Z\right)/\left(\Z/n\Z\right)^{\ast}
\end{equation}
Note that 
\begin{equation}
  \label{eq:12.27}
  m \equiv 0 \pmod {n} \Longrightarrow F_m\phi_n= \varphi(n)\cdot \phi_1
\end{equation}
and we get by continuity an action of the multiplicative monoid of ``super-natural-number''
\begin{equation}
  \label{eq:12.28}
  \hat{\Z}/\hat{\Z}^{\ast}=\prod_{p \text{ prime}}p^{\N\cup \{\infty\} }
\end{equation}
by ring endomorphisms of $\mathcal{W}$. In particular, we get the homomorphism
\begin{equation}
  \label{eq:12.29}
  F_0=\lim_{m\to 0\in\hat{\Z}}F_m\; : \; \mathcal{W}\longrightarrow \Z. 
\end{equation}
We have $F_0\circ F_m=F_0$ for all $m$, and
\begin{equation}
  \label{eq:12.30}
  F_0\phi_n = \# \bbmu_n^{\ast}=\varphi(n)
\end{equation}
is Euler's function. \vspace{.1cm}\\
The group of roots of unity $\bbmu_\infty$ give rise to the group-ring $\Z\cdot\bbmu_\infty$, which maps homomorphically onto the ring of cyclotomic integers $\Z[\bbmu_\infty]$, and also maps onto $\Z$ via the augmutation. Taking $\hat{\Z}^*=\text{Aut}(\bbmu_\infty)$ invariants we get
\begin{equation}
  \label{eq:13.18}
  \begin{aligned}
  \begin{tikzpicture}[baseline=0mm]
    \node at (12mm,15mm) {$\mathcal{W}=(\Z\cdot \bbmu_\infty)^{\hat{\Z}^*}$};
    \draw[->>] (25mm,16mm)--(40mm,23mm);
    \draw[->>] (25mm,14mm)--(40mm,7mm);
    \node at (32mm,22mm) {$t_1$};
    \node at (35mm,12mm) {$F_0$};
    \node at (43mm,6mm) {$\Z$};
    \node at (53mm,25mm) {$\left(\Z[\bbmu_\infty]\right)^{\hat{\Z}^*}=\Z$};
  \end{tikzpicture}
    \end{aligned}
\end{equation}
\begin{remark}
  We have
  \begin{equation*}
    \cring(\mathcal{W},\C) \equiv \cring(\mathcal{W},\Z)\equiv \hat{\Z}/\hat{Z}^*
  \end{equation*}
  Indeed, from (\Ref{eq:12.18}) we obain
  \begin{equation*}
    \cring(\mathcal{W},\C)\equiv \prod_p\cring(\mathcal{W}_{p^\infty},\C). 
  \end{equation*}
  From (\Ref{eq:12.16}), we obtain for $\psi\in\cring(\mathcal{W}_{p^\infty},\C)$ the implications for $n\ge 1$,
  \begin{equation*}
    \begin{array}[h]{lll}
      \ & \ & \displaystyle \psi(\phi_{p^n})\not= 0 \\\\
      \displaystyle \Longrightarrow & \ & \psi(\phi_{p^m}) = p^m\left(1-p^{-1}\right) \text{ for $m=1,2,\ldots , n-1$} \\\\
      \displaystyle \Longrightarrow & \ & \psi(\phi_{p^n}) = p^n\left(1-p^{-1}\right) \text{ or $\psi(\phi_{p^n})=-p^{n-1}$} 
    \end{array}
  \end{equation*}
  From this it follows that $\psi=\text{tr}\circ F_{p^n}$,  $n\ge 0$,
  \begin{equation*}
    \text{tr}\circ F_{p^n}(\phi_{p^m})=
\left\{
    \begin{array}[h]{ll}
      1 & m=0 \\
      p^m(1-p^{-1}) & 1\le m \le n \\
      -p^n & m=n+1 \\
      0 & m>n+1
    \end{array}
    \right. 
  \end{equation*}
or that $\psi=\lim\limits_{n\to\infty} \text{tr}\circ F_{p^n} = F_{p^\infty}=F_0$. Thus
    \begin{equation*}
      \cring(\mathcal{W}_{p^\infty},\C) \equiv p^{\N\cup\{\infty\}}
\end{equation*}
\qed
\end{remark}
We have as well the additive projection
\begin{equation}
  \label{eq:12.31}
  \int \; : \; \mathcal{W}\rightarrow \Z \quad , \quad \int \phi_n=
\left\{  \begin{array}[h]{ll}
           1 & n=1 \\
           0 & n>1
\end{array}\right.
\end{equation}
and we have
\begin{equation}
  \label{eq:12.32}
  \int (\phi_{n_1}\cdot \phi_{n_2}) =
\left\{  \begin{array}[h]{ll}
           \varphi(n) & n_1=n_2=n \\\\
           0 & n_1\not= n_2
  \end{array}\right. 
\end{equation}
We can extend scalar to $\C$, and obtain the ring homomorphisms
\begin{equation}
  \label{eq:12.33}
  t_m=\text{tr}\circ F_m: \mathcal{W}_{\C}=\C\otimes \mathcal{W}\longrightarrow \C. 
\end{equation}
The projection $\int: \mathcal{W}_{\C}\rightarrow \C$ can be expanded using the $t_m$'s as
\begin{equation}
  \label{eq:12.34}
  \int f \equiv \lim_{M\to 0\in \hat{\Z}}\frac{1}{M}\sum_{m=1}^M t_m(f).
\end{equation}
We can complete $\mathcal{W}_{\C}$ with respect to the state $\int$ and obtain the Hilbert space
\begin{equation}
  \label{eq:12.35}
  \mathcal{H} = \hat{\mathcal{W}}_{\C}\quad, \quad \langle f,g\rangle = \int (f\cdot \overline{g})
\end{equation}
with orthogonal basis $\{\phi_n\}_{n\ge 1}\;\; , \;\; \|\phi_n\|^2=\varphi(n)$. \\
The complete-multiplicativity of the $F_m$, (\ref{eq:2.19}), suggests considering
  \begin{equation}
    \label{eq:12.36}
    \zeta (F,s):= \sum_{m\ge 1}\frac{1}{m^s} F_m\equiv \prod_p\left(1+\frac{1}{p^s}F_p+\frac{1}{p^{2s}}F_{p^2}+\cdots\right)\equiv
      \prod_p \frac{1}{(1-p^{-s}F_p)}
    \end{equation}
    The adoint of $F_m$, the ``verschiebung'' $F_m^\ast$ is the
    additive map given by \break  $F_m^\ast\phi_n=\phi_{m\cdot n}$. The
    multiplicativity $F^\ast_{m_1}\circ F_{m_2}^\ast=F_{m_1\cdot m_2}^{\ast}$ suggest considering similarly
    \begin{equation}
      \label{eq:12.37}
      \zeta(F^{\ast},s)=\sum\limits_{m\ge 1}\frac{1}{m^s}F_m^{\ast} = \prod_p \frac{1}{(1-p^{-s}F_p^{\ast})}=\zeta(F,\overline{s})^{\ast}
    \end{equation} 
    In terms of these zeta elements we can interpret
    Ramanujan's sums
\begin{equation}
  \label{eq:12.38}
  t_m\big(\zeta(F^{\ast},t)\phi_1\big) = \text{tr}\big(F_m\zeta(F^{\ast},t)\phi_1\big) = \sum_{n\ge 1}\frac{C_n^m}{n^t} = \frac{1}{\zeta(t)}\cdot\left(
    \sum_{d\mid m}d^{1-t}\right) 
\end{equation}

\begin{equation}
  \label{eq:12.39}
\text{tr}\left(\zeta (F,s)\phi_n\right) = \sum_{m\ge 1}\frac{C_n^m}{m^s}=\zeta(s)\cdot\left(\sum_{d\mid n}\mu\left(\frac{n}{d}\right) d^{1-s}\right)
\end{equation}

\begin{equation}
  \label{eq:12.40}
\text{tr}\left(\zeta(F,s)\zeta(F^{\ast},t)\phi_1\right) = \sum_{n,m\ge 1}\frac{C_n^m}{n^tm^s}= \frac{\zeta(s)}{\zeta(t)}\cdot \zeta (s+t-1)
\end{equation}
with $\zeta(s)$ the Riemann zeta function. 

  The probability measure $\frac{1}{M}\sum\limits_{m=1}^M\delta_m$ on $\hat{\Z}$, converges as $M\to 0 \in \hat{\Z}$, to the Haar measure
  $dz$, and we get an isomorphism, the Fourier transform
  \begin{equation}
    \label{eq:12.41}
    \begin{array}[h]{rcl}
      \mathcal{H}=\hat{\mathcal{W}}_{\C} &\xrightarrow{\qquad \sim  \qquad} & L_2(\hat{\Z},dz)^{\hat{\Z}^{\ast}} \\\\
      f &\xmapsto{\qquad\quad\qquad} &  \hat{f}(z):= \text{tr}(F_zf) \\\\
      F_{m } &\xmapsto{\qquad\quad\qquad} & \hat{F}_{m}\hat{f}(z):= \hat{f}(mz) \\\\
      F_m^{\ast} &\xmapsto{\qquad\quad\qquad} & \hat{F}_m^{\ast}\hat{f}(z) := m\cdot \hat{f}(z/m) \\\\
     \langle f,g \rangle \hspace{-9mm} &=& \hspace{-9mm} \displaystyle\int\limits_{\hat{\Z}}\hat{f}(z)\cdot \overline{\hat{g}(z)}\, dz

    \end{array}
  \end{equation}


  The ring $\mathcal{W}$ has a structure of a special-$\lambda$-ring, with $\lambda$-operations.
  \begin{equation}
    \label{eq:13.32}
    \lambda^n:\mathcal{W}\to\mathcal{W}\quad, \quad \lambda^0(x)=1\quad ,\quad
    \lambda^1(x)=x\quad, \quad \lambda^n(x_1+x_2)=\sum\limits_{n_1+n_2=n}\lambda^{n_1}(x_1)\cdot\lambda^{n_2}(x_2)
  \end{equation}
  defined by
  \begin{equation}
    \label{eq:13.33}
    \begin{array}[h]{rll}
      \lambda_t(x)=\sum_{n\ge 0} (-1)^n\lambda^n(x)\cdot t^n: &\mathcal{W}\xrightarrow{\qquad} 1+ &\hspace{-3mm} t\cdot\mathcal{W}[|t|]_{\text{rat}} \vspace{-.3cm}\\
      \ & \rotatebox{270}{\huge $\supseteq$} & \rotatebox{270}{\huge $\supseteq$} \vspace{.1cm}\\
      \oplussum\limits_{n\ge 0}\N\cdot \phi_n \equiv &\mathcal{W}^+\xrightarrow{\qquad} 1+\hspace{-.5cm} &t\cdot\mathcal{W}[t]
    \end{array}
\end{equation}
\begin{equation}
  \label{eq:13.34}
  \lambda_t(\phi_n)=\prod\limits_{\xi\in\bbmu_n^*} (1-[\xi]t)=\sum\limits_{m=0}^{\varphi(n)}\lambda^m(\phi_n)\cdot (-t)^m
\end{equation}

\begin{equation}
  \label{eq:13.35}
\tag*{(13.34) E.g.} \stepcounter{equation}
  \begin{array}[h]{l}
    \lambda_t(\phi_1)=1-t \\\\
    \lambda_t(\phi_2) = 1-t\cdot \phi_2 \\\\
    \lambda_t(\phi_3) = 1 - t\phi_3+t^2 \\\\
    \lambda_t(\phi_4) = 1-t\phi_4+t^2 \\\\
    \lambda_t(\phi_5)=1-t\phi_5+t^2(\phi_5+2\phi_1)-t^3\phi_5+t^4 \\\\
    \lambda_t(\phi_6)=1-t\phi_6+t^2 \\\\
    \lambda_t(\phi_7)=1-t\phi_7+t^2(2\cdot\phi_7+3\phi_1)-t^3(3\phi_7+2\phi_1)+t^4(2\phi_7+3\phi_1)-t^5\phi_7+t^6 \\\\
    \lambda_t(\phi_8)=1-t\phi_8+t^2(\phi_4+2\phi_2+2\phi_1)-t^3\phi_8+t^4 \\\\
    \lambda_t(\phi_9)=1-t\phi_9+t^2(\phi_9+3\phi_3+3\phi_1)-t^3(3\phi_9+\phi_3)+t^4(\phi_9+3\phi_3+3\phi_1)-t^5\phi_9+t^6 \\\\
    \lambda_t(\phi_{10})=1-t\phi_{10}+t^2(\phi_5+2\phi_1)-t^3\phi_{10}+t^4
  \end{array}
\end{equation}
We have,
\begin{equation}
  \label{eq:13.36}
  \lambda_t(\phi_n)=t^{\varphi(n)}\cdot \lambda_{t^{-1}}(\phi_n)\qquad \text{for $n>2$}
\end{equation}
\begin{equation}
  \label{eq:13.37}
  \text{tr}\big(\lambda_t(\phi_n)\big) = \prod\limits_{\xi\in\bbmu_n^*} (1-\xi\cdot t) = \tilde{\phi}_n(t)=\phi_n(t)
\end{equation}
\begin{equation}
  \label{eq:13.38}
  F_0\big(\lambda_t(\phi_n)\big)=(1-t)^{\varphi(n)}
\end{equation}
\begin{equation}
  \label{eq:13.39}
  \lambda^n(k\cdot\phi_1) = \begin{pmatrix} k \\ n  \end{pmatrix} \quad \text{satisfies}\quad (-1)^n\lambda^n(-k\cdot
  \phi_1) =  \begin{pmatrix} k+n-1 \\ n  \end{pmatrix} = \lambda^n(k+n-1)
\end{equation}
We denote the Grothendieck involution by
\begin{equation}
  \label{eq:13.42}
  D(t)=\begin{bmatrix} 1 & 0 \\ 1 & -1 \end{bmatrix}(t) = \frac{t}{t-1}= \frac{1}{1-t^{-1}}
\end{equation}
so that $DD(t)=t$, $D(0)=0$, $D(1)=\infty$, $D(2)=2$, and in the notations of \S \Ref{sec:3}, $1/p+1/q=1 \Leftrightarrow
q=D(p)$. We obtain Grothendieck's gamma-operations
\begin{equation}
  \label{eq:13.43}
  \begin{array}[h]{l}
    \gamma_t(x)=\sum\limits_{n\ge 0} (-1)^n\gamma^n(x)t^n:\; \mathcal{W}\xrightarrow{\quad} 1+t\mathcal{W}[|t|]_{\text{rat}} \\\\
    \gamma_t^0(x)=1,\quad -\gamma^1(x)=x,\quad \gamma^n(x_1+x_2)=\sum\limits_{n_1+n_2=n}\gamma^{n_1}(x_1)\cdot\gamma^{n_2}(x_2)
  \end{array}
\end{equation}
by setting
\begin{equation}
  \label{eq:13.44}
  \gamma_t(x)=\lambda_{D(t)}(x) \Leftrightarrow (-1)^n\gamma^n(x)=\lambda^n(x+n-1)\quad, \text{cf. (~\Ref{eq:13.39})}
  \end{equation}
  We have
  \begin{equation}
    \label{eq:13.45}
    \mathcal{W}=\Z\cdot\phi_1\oplus I \quad ,\quad I=\ker F_0=\oplussum\limits_{n>1}\Z(\phi_n-\varphi(n)\cdot \phi_1)
  \end{equation}
  The gamma-filteration is given by $I_0=\mathcal{W}$, $I_1=I$, and
  \begin{equation}
    \label{eq:13.46}
    I_n\equiv \text{additive subgroup of $\mathcal{W}$ generated by $\gamma^{n_1}(a_1)\cdots \gamma^{n_\ell}(a_\ell), 
      a_i\in I , \quad \sum n_i\ge n$}
  \end{equation}
  It is a $\lambda$-filtration, and on the associated graded ring
  \begin{equation}
    \label{eq:13.47}
    \text{gr}_\gamma\mathcal{W} \equiv \oplussum_{n\ge 0} I_n/I_{n+1}
  \end{equation}
  we have
  \begin{equation}
    \label{eq:13.48}
   \ \hspace{24mm} F_m \equiv \oplussum\limits_{n\ge 0} m^n \qquad \text{c.f. \cite{AT68} prop. 5.3.} 
  \end{equation}
  and also
  \begin{equation}
    \label{eq:13.49}
   \ \hspace{14mm}    (-1)^{m+1}\lambda^m \equiv \oplussum\limits_{n\ge 1} m^{n-1} \qquad \text{c.f. \cite{AT68} prop. 5.5.} 
  \end{equation}
Note that
  \begin{equation}
    \label{eq:13.50}
    \   \ \hspace{22mm} \gamma_t:I^+\equiv \oplussum\limits_{n>1}\N\cdot (\phi_n-\varphi(n)\phi_n)\xrightarrow{\quad} 1+t\mathcal{W}[t]
  \end{equation}
  and
  \begin{equation}
    \label{eq:13.51}
    \gamma_t(\phi_n-\varphi(n)\phi_1)=(1-t)^{\varphi(n)}\cdot \prod\limits_{\xi\in\bbmu_n^*} \big( 1-[\xi]\cdot\frac{t}{t-1}\big) = \prod\limits_{\xi\in\bbmu_n^*}\big( 1- (1-[\xi])t\big)
  \end{equation}
  i.e. passing from $\lambda^m(\phi_n)$ to $\gamma^m(\phi_n-\varphi(1)\phi_1)$ amounts to passing from the $m$-th elementary symmetric function of the roots of unity $\bbmu_n^*$, to that of the ``cyclomic units'' $1-[\xi]$, $\xi\in\bbmu_n^*$.
\begin{remark}
One can proceed as in \cite{D}, inverting the $F_{n}$'s to get an action of the
multiplciative group of positive rational numbers $\Q^{+}=\oplus p^{\Z}$, the
free abelian group on the primes, and extend scalars along $\Q^{+}\subseteq
\R^{+}$; the infinitesimal generator of the $\R^{+}$ action should give an
interesting operation when acting on appropriate cohomology groups.\\
Viewed from the perspective of the Fourier transform (\Ref{eq:12.41}),
``inverting the $F_n$'s'' amount to passing from $\hat{\Z}/\hat{\Z}^*$
to
$\mathbb{Q}\underset{\Z}{\otimes}
\hat{\Z}/\hat{Z}^*=A_{\text{fin}}/\hat{\Z}^*$; ``extending scalars
along $\mathbb{Q}^+\subseteq\R^{+}$'' amount essentially to the passage to
$\mathbb{A}_\text{fin}\oplus\R/_{\hat{\Z}^*\mathbb{Q}^*}=\mathbb{A}/_{\mathbb{Q}^*\hat{\Z}^*}$.
This space and the action of
$\mathbb{A}^*/_{\mathbb{Q}^*\hat{\Z}^*}\equiv\R^+$ on it, was first
suggested by Weil \cite{weil1966fonction}, and was farther studied in \cite{Har90}, \cite{Har01}, and by
Connes \cite{connes1999trace}. But note that for a curve $X$, over a
finite field $\mathbb{F}_q$, we do \underline{not} have a proof of the
Riemann-hypothesis by simply using the action of $A^*_{K}/_{K^*}$ on $A_{K}/_{K^*}$, $K=\mathbb{F}_g(X)$ the function field of $X$; we realy
need the intersection theory on the surface $X\times X$. 
\label{remark:12.42}
\end{remark}

\nocite{*}

\bibliography{Non-Additive-Geometry-and-Frobenius-Correspondences}{}
\bibliographystyle{alpha}
\end{document}